\documentclass[letterpaper,reqno,12pt]{amsart} 
\usepackage[margin=1.2in]{geometry}

\usepackage{amssymb, enumerate, color}
\usepackage{mathrsfs}
\usepackage[all]{xy}
\usepackage{hyperref}
\usepackage{enumitem}
\usepackage{graphicx}
\usepackage{subcaption}
\usepackage{here}

\usepackage{tikz}

\numberwithin{equation}{section} 

\theoremstyle{plain}
\newtheorem{thm}{Theorem}[section] 
\newtheorem{cor}[thm]{Corollary}
\newtheorem{prop}[thm]{Proposition}

\newtheorem{lem}[thm]{Lemma}
\newtheorem*{mainthm}{Theorem A}

\theoremstyle{definition} 
\newtheorem{defn}[thm]{Definition}
\newtheorem{lem-defn}[thm]{Lemma-Definition}

\newtheorem{eg}[thm]{Example} 
\newtheorem{question}[thm]{Question}
\newtheorem*{notation}{Notation}
\newtheorem{notation2}[thm]{Notation}

\theoremstyle{remark}
\newtheorem{rem}[thm]{Remark}

\newtheorem*{acknowledgement}{Acknowledgments}

\def\ge{\geqslant}
\def\le{\leqslant}
\def\phi{\varphi}
\def\epsilon{\varepsilon}
\def\tilde{\widetilde}

\def\mapsto{\longmapsto}

\def\onto{\relbar\joinrel\twoheadrightarrow}

\newcommand{\sO}{\mathcal{O}}

\newcommand{\N}{\mathbb{N}}
\newcommand{\Q}{\mathbb{Q}} 
 
\newcommand{\R}{\mathbb{R}} 
\newcommand{\Z}{\mathbb{Z}}
\newcommand{\PP}{\mathbb{P}}

\newcommand{\m}{\mathfrak{m}}
\newcommand{\n}{\mathfrak{n}}

\newcommand{\ch}{\mathrm{ch}}
\newcommand{\Frac}{\mathrm{Frac}}

\newsavebox{\circlebox}
\savebox{\circlebox}{\fontencoding{OMS}\selectfont\Large\char13}
\newlength{\circleboxwdht}

\def\Spec{\operatorname{Spec}}

\def\Pic{\operatorname{Pic}}

\def\Exc{\operatorname{Exc}}

\def\ord{\operatorname{ord}}

\def\id{\operatorname{id}}

\newcommand{\dq}[1]{\lq\lq\textup{#1}\rq\rq}

\title{General hyperplane sections of log canonical threefolds in positive characteristic}

\author{Kenta Sato}
\address{Faculty of Mathematics, Kyushu University, 744 Motooka, Nishi-ku, Fukuoka 819-0395, Japan}
\email{ksato@math.kyushu-u.ac.jp}



\begin{document}

\begin{abstract}
In this paper, we prove that if a $3$-dimensional quasi-projective variety $X$ over an algebraically closed field of characteristic $p>3$ has only log canonical singularities, then so does a general hyperplane section $H$ of $X$.
We also show that the same is true for klt singularities, which is a slight extension of \cite{ST20}.
In the course of the proof, we provide a sufficient condition for log canonical (resp.~klt) surface singularities to be geometrically log canonical (resp.~geometrically klt) over a field.
\end{abstract}

\maketitle
\markboth{K.~SATO}{GENERAL HYPERPLANE SECTIONS OF LOG CANONICAL THREEFOLDS}


\section{Introduction}

The classical Bertini theorem states that if a projective scheme $X$ over an algebraically closed field is smooth, then a general hyperplane section $H$ of $X$ is also smooth.
Many variants of this theorem have been established; for instance, it is known that if $X$ is reduced (resp.~normal, Cohen-Macaulay, Gorenstein), then so is $H$.
Furthermore, in characteristic zero, the argument presented in Reid's paper (\cite{Rei}) implies that certain classes of singularities in the minimal model program possess a similar property. 
Specifically, if $X$ has only log canonical (resp. klt, canonical, or terminal) singularities, then the same property holds for $H$.
In this paper, we consider the case of positive characteristic, and pose the following question:

\begin{question}\label{Bertini question}
Let $k$ be an algebraically closed field of characteristic $p>0$ and $X \subseteq \PP^N_k$ be a subvariety of $\PP^N_k$.
If $X$ has only log canonical (resp.~klt, canonical, terminal) singularities, does a general hyperplane section $H$ have the same properties?
\end{question}

If $X$ is two-dimensional, then the assertion in Question \ref{Bertini question} holds true, as $X$ has only isolated singularities.
Likewise, if $X$ is three-dimensional with terminal singularities, then the question is confirmed for the same reason.
However, the remaining cases, even in dimension three, are more intricate.
Note that Reid's argument is not applicable in positive characteristic due to the Bertini theorem's failure for base point free linear systems.
Recently, an affirmative answer has been provided in \cite{ST20} using jet schemes and $F$-singularities, in the case where $X$ has canonical or klt singularities and is three-dimensional, while requiring $p$ to be greater than $5$ in the klt case.

In this paper, we give an affirmative answer to the log canonical case and klt case in dimension three with $p>3$.

\begin{mainthm}[Corollary \ref{bertini}]
Let $X \subseteq \PP^N_k$ be a three-dimensional normal quasi-projective variety over an algebraically closed field $k$ of characteristic $p>3$.
If $X$ has only log canonical (resp.~klt) singularities, then so does a general hyperplane section $H$ of $X$.
\end{mainthm}

It is worth noting that the result is novel even in the klt case, as in \cite{ST20}, the case where $p=5$ was beyond the scope of the investigation.

Our approach to prove Theorem A relies on an idea of \cite{CGM}. 
We examine the closed subscheme $\mathcal{X}$ of the product $X \times_k (\PP^N_k)^*$ whose closed fiber over $(\PP^N_k)^*$ at a closed point $H \in (\PP^N_k)^*$ coincides with the hyperplane section $X \cap H$.
To verify that a general closed fiber of $\mathcal{X}$ is log canonical (resp.~klt), we need to show that the generic fiber $\mathcal{X}_{\eta}$ is geometrically log canonical (resp.~geometrically klt) over the function field $K((\PP^N_k)^*)$. (See Proposition \ref{family} for more details).
On the other hand, since the natural map $p : \mathcal{X}_{\eta} \to X$ has geometrically regular fibers, if $X$ is log canonical (resp.~klt), then so is $\mathcal{X}_{\eta}$.
Furthermore, we also know that $\mathcal{X}_{\eta}$ is geometrically normal over $K((\PP^N_k)^*)$ (\cite{CGM}).
This leads to the following question:

\begin{question}
Let $S$ denote a two-dimensional geometrically normal variety over a field $K$ that is not necessarily perfect, and suppose that $S$ has log canonical (resp.~klt) singularities. 
What is a sufficient condition for $S$ to be geometrically log canonical (resp.~geometrically klt) over $K$?
\end{question}

It should be noted that in characteristic $2$ or $3$, there exist instances of geometrically normal surface singularities $(s \in S)$ over a field $K$ with the residue field isomorphic to $K$ that are log canonical (resp.~klt) but not geometrically log canonical (resp.~geometrically klt) over $K$ (Example \ref{counterexample}).
However, in Section \ref{4}, it is observed that such examples do not occur in characteristic larger than $3$.
To be precise, we prove the following theorem.

\begin{mainthm}[Corollary \ref{geometric klt}, Corollary \ref{geometric lc}]
Let $X$ be a $2$-dimensional geometrically normal variety over a field $k$ of characteristic $p>3$ and $x_1, \dots, x_n \in X$ be the non-smooth points of $X$ over $k$.
Suppose that $X$ is log canonical (resp.~klt) at $x_i$ and the residue field $\kappa(x_i)$ of $X$ at $x_i$ is separable over $k$ for every $i$.
Then $X$ is geometrically log canonical (resp.~geometrically klt) over $k$.
\end{mainthm}

By a standard argument, Theorem A follows from Theorem B.

\begin{acknowledgement}
The author wishes to express his gratitude to Professor Shunsuke Takagi for his valuable advice and suggestions.
He is also grateful to Professors Teppei Takamatsu, Shou Yoshikawa and Tatsuro Kawakami for helpful comments.
He also thanks the anonymous referee for the comments and suggestions that helped to improve the paper.
This work was supported by JSPS KAKENHI Grant Number 20K14303.
\end{acknowledgement}

\begin{notation}
Throughout this paper, all rings are assumed to be commutative and with unit element and all schemes are assumed to be Noetherian and separated. 
In this paper, a \dq{variety} over a field $k$ is an integral separated scheme of finite type over $k$. 
\end{notation}

\section{Preliminaries}

This section provides preliminary results needed for the rest of the paper. 

\subsection{Singularities in the MMP}
In this subsection, we recall the definition and basic properties of singularities in minimal model program (or MMP for short). 

Throughout this subsection, unless otherwise stated, $X$ denotes an excellent normal integral scheme with a dualizing complex $\omega_X^\bullet$. 
The \textit{canonical sheaf} $\omega_X$ associated to $\omega_X^{\bullet}$ is the coherent $\sO_X$-module defined as the first nonzero cohomology of $\omega_X^\bullet$.
A \textit{canonical divisor} of $X$ associated to $\omega_X^\bullet$ is any Weil divisor $K_X$ on $X$ such that $\sO_X(K_X) \cong \omega_X$. 

\begin{rem}\label{rem: dualizing}
While there is ambiguity in the choice of a dualizing complex or a canonical divisor, there are often standard choices as follows.
\begin{enumerate}[label=\textup{(\arabic*)}]
\item If $X$ is of finite type over a field $k$ with the structure map $\pi : X \to \Spec k$, then we always set $\omega_X^{\bullet} : = \pi^! (\sO_{\Spec k })$.
By adopting this convention, the following property holds by \cite[Lemma 0EA0]{Sta}:
if $X$ is a geometrically normal (cf.~Definition \ref{def:geom} below) variety over a field $k$, then we have 
\[
\omega_X \cong (\Omega_{X/k}^{\wedge \dim X})^{**},\] 
where $^{**}$ denotes the reflexive hull.
\item We fix a canonical divisor $K_X$ of $X$ associated to $\omega_X^\bullet$, and given a proper birational morphism $\pi:Y \to X$ from a normal integral scheme $Y$, we always choose a canonical divisor $K_Y$ of $Y$ that is associated to $\pi^! \omega_X^{\bullet}$ and coincides with $K_X$ outside the exceptional locus $\mathrm{Exc}(f)$ of $f$. 
\end{enumerate}
\end{rem}

\begin{defn}
A proper birational morphism $f: Y \to X$ from a regular integral scheme $Y$ is said to be a \textit{resolution of singularities} of $X$. 
A resolution $f:Y \to X$ is said to be a \textit{log resolution} if the exceptional locus $\mathrm{Exc}(f)$ of $f$ is a simple normal crossing divisor.  
\end{defn}

First, we give the definition of singularities in the MMP that makes sense in arbitrary characteristic.
\begin{defn}\label{mmp}
Suppose that $X$ is $\Q$-Gorenstein (that is, $K_X$ is $\Q$-Cartier).
\begin{enumerate}[label=(\roman*)]
\item
Given a proper birational morphism $f:Y \to X$ from a normal integral scheme $Y$, we write
\[
K_{Y/X} : = K_Y- f^*(K_X).
\]
For each prime divisor $E$ on $Y$, the \textit{discrepancy} $a_{E}(X)$ of $X$ at $E$ is defined as
\[
a_{E} (X) : = \ord_E( K_{Y/X}). 
\]
\item
We say that $X$ is \textit{log canonical} (resp.~\textit{klt}) at a point $x \in X$ if $a_E(\Spec \sO_{X,x}) \ge -1$ (resp.~$ >-1$) for every proper birational morphism $f:Y \to \Spec \sO_{X,x}$ from a normal integral scheme $Y$ and for every prime divisor $E$ on $Y$.
We say that $X$ is \textit{log canonical} (resp.~\textit{klt}) if it is log canonical (resp.~klt) for every $x \in X$.
\end{enumerate}
\end{defn}

\begin{rem}\label{rmk on discrepancy}
Suppose that $X$ is $\Q$-Gorenstein and there exists a log resolution $f: Y \to X$.
Then for a point $x \in X$, the following are equivalent
\begin{itemize}
\item $X$ is log canonical (resp.~klt) at $x$.
\item $a_E(X) \ge -1$ (resp.~$>-1$) for all prime divisor $E$ in $Y$ whose center $f(E)$ contains $x$.
\item There exists an open neighborhood $U \subseteq X$ of $x$ such that $U$ is log canonical (resp.~klt).
\end{itemize}
\end{rem}

Given a property $\mathcal{P}$ of local rings, we say that a scheme $X$ satisfies $\mathcal{P}$ at $x \in X$ if the stalk $\sO_{X,x}$ satisfy $\mathcal{P}$.

\begin{defn}\label{def:geom}
Let $X$ be a scheme over a field $k$ and $x \in X$ be a point.
We say that $X$ is \emph{geometrically $\mathcal{P}$} at $x$ if for every finite field extension $\ell$ of $k$, the base change $X_{\ell} : = X \times_{k} \ell$ satisfies $\mathcal{P}$ at any point $y \in X_{\ell}$ lying over $x$.
\end{defn}

\begin{rem}\label{rmk on geometrically}
Let $X$ be a scheme of finite type over a field $k$ and we set $\mathcal{P} : = $\dq{regular} (resp.~\dq{$(R_n)$} for some fixed $n$, \dq{normal}, \dq{reduced}, \dq{irreducible}, \dq{integral}).
It is known that the following are equivalent:
\begin{enumerate}[label=\textup{(\alph*)}]
\item $X$ is geometrically $\mathcal{P}$ over $k$.
\item $X \times_k \ell$ satisfies $\mathcal{P}$ for every finitely generated field extension $\ell$ of $k$.
\item $X \times_k \ell$ satisfies $\mathcal{P}$ for every field extension $\ell$ of $k$.
\item $X \times_k \overline{k}$ satisfies $\mathcal{P}$, where $\overline{k}$ is the algebraic closure of $k$.
\end{enumerate}
We refer to \cite[Lemma 037K]{Sta} when $\mathcal{P}=$\dq{irreducible}, and to \cite[Lemma 038U, 038X and 02V4]{Sta} when $\mathcal{P}=$\dq{$(R_n)$}.
\end{rem}

\begin{defn}
Let $(Y, D)$ be a pair of a Noetherian integral scheme $Y$ over a field $k$ and a reduced divisor $D$ on $Y$.
We say that $(Y,D)$ is \emph{geometrically SNC} over $k$ if for every finite field extension $\ell$ of $k$, the pair $(Y_{\ell}, D_{\ell})$ of the base change $Y_{\ell} : = Y \times_{k} \ell$ and the flat pullback $D_{\ell}$ of $D$ is a SNC pair.
\end{defn}

\begin{lem}\label{criterion1}
Let $X$ be a variety over a field $k$ and $f: Y \to X$ be a log resolution of $X$ such that the pair $(Y , \Exc(f))$ is geometrically SNC over $k$.
If $X$ is $\Q$-Gorenstein and klt (resp.~log canonical), then $X$ is geometrically klt (resp.~geometrically log canonical) over $k$.
\end{lem}

\begin{proof}
Let $\ell $ be a finite field extension of $k$.
We first show that $X_{\ell}$ is normal and $\Q$-Gorenstein.
Since $\ell$ is flat over $k$, if $X$ satisfies the Serre condition $(S_m)$ for some $m >0$, then so is the base change $X_{\ell} : = X \times_k \ell$.
On the other hand, since $Y$ is smooth over $k$, so is $X$ around the locus where $f$ is isomorphic.
Therefore, $X_{\ell}$ satisfies the $(R_1)$-condition, too.
We also note by \cite[Lemma 0E9U]{Sta} and Remark \ref{rem: dualizing} (1) that one has $\nu^* \omega_{X} \cong \omega_{X_{\ell}}$, here $\nu: X_{\ell} \to X$ is the natural morphism.
Therefore, $X_{\ell}$ is $\Q$-Gorenstein.

In order to show that $X_{\ell}$ is klt (resp.~log canonical), we consider the following Cartesian diagram:
\[
\xymatrix{
Y_{\ell} \ar^-{f_\ell}[r] \ar^-{\mu}[d] & X_{\ell} \ar^-{\nu}[d] \ar[r] & \Spec (\ell) \ar[d] \\
Y \ar^-{f}[r] & X \ar[r] & \Spec (k)
}\]
We note that $f_{\ell}$ is a log resolution of $X_{\ell}$ by assumption.
By Remark \ref{rmk on discrepancy}, it is enough to show that $a_E(X_{\ell})>-1$ (resp.~$\ge -1$) for every exceptional prime divisor $E \subseteq Y_{\ell}$.
Let $F : = \mu(E) \subseteq Y$ be the image of $E$ by the finite morphism $\mu$.
Since $F$ is an exceptional prime divisor over $X$, it follows from the assumption that $F$ is smooth over $k$.
Therefore, the coefficient of $\mu^*F$ at $E$ is one.
Combining this with the equation
\[
K_{Y_{\ell}/X_{\ell}} = \mu^*(K_{Y/X}),
\]
we conclude that $a_F(X_{\ell}) = a_E(X) > -1$ (resp.~$\ge -1$), as desired.
\end{proof}

\subsection{Curves}

In this paper, a curve is a (not necessarily regular) integral projective scheme $C$ over a field $k$ with $\dim C=1$.

For a curve $C$ over $k$, the \emph{arithmetic genus} $g(C)$ is defined by
\[
g(C) : = \frac{ \dim_k(H^1(C, \sO_C))}{\dim_k(H^0(C, \sO_C))} \in \N_{\ge 0}.
\]
For a coherent sheaf $\mathcal{F}$ on $C$, we write
\[
\chi(C/k, \mathcal{F}) : = \dim_k (H^0(C, \mathcal{F})) -\dim_k (H^1(C, \mathcal{F}))
\]
For a Weil divisor $D = \sum_{i=1}^n a_i P_i$ on $C$, we define the \emph{degree} of $D$ over $k$ by
\[
\deg_{C/k}(D) : = \sum_{i=1}^n a_i \dim_k ( \kappa(P_i)) \in \Z,
\]
where $\kappa(P_i)$ is the residue field of $C$ at $P_i$.
This defines the homomorphism
\[
\deg_{C/k} : \Pic(C) \to \Z.
\]
See \cite[Subsection 1.4]{Ful} for more details.

We recall the Riemann-Roch theorem on a curve.
\begin{lem}\label{RR}
Let $C$ be a (not necessarily regular) curve over a field $k$ and $L$ be an invertible sheaf on $C$.
Then the following hold.
\begin{enumerate}[label=$(\arabic*)$]
\item $\chi(C/k, L) = \deg_{C/k}( L ) + \chi(C/k, \sO_C)$.
\item If we have $H^0(C, \sO_C)=k$ and $C$ is Gorenstein, then we have 
\[
\deg_{C/k}(\omega_C) = 2 g(C) -2
\]
\end{enumerate}
\end{lem}

\begin{proof}
The assertion in (1) follows from \cite[Example 18.3.4]{Ful}.
For (2), it follows from Serre duality that 
\[
\chi(C/k, \omega_C) = -\chi(C/k, \sO_C) =g(C) -1
\]
Then the assertion follows from (1).
\end{proof}

\subsection{Surface Singularities and dual graphs}

\begin{defn}
We say that $(x \in X)$ is a \emph{normal surface singularity} if $X=\Spec R$ is an affine scheme with a dualizing complex where $R$ is a two-dimensional excellent normal local ring and if $x$ is the unique closed point of $X$. 
\end{defn}

First we recall the definition and basic properties of intersection numbers.
Let $(x \in X)$ be a normal surface singularity and $f :Y \to X$ be a proper birational morphism from a normal integral scheme $Y$ with the exceptional divisor $\Exc(f)=\sum_i^n E_i$. 
For a Cartier divisor $D$ on $Y$ and an exceptional Weil divisor $Z= \sum_{i} a_i E_i$ on $Y$, we define the intersection number $(D \cdot Z)$ as follows: 
\begin{eqnarray*}
  (D \cdot Z)= \sum_{i} a_i \deg_{E_i/\kappa(x)}(\sO_Y(D) |_{E_i}),
\end{eqnarray*}
where $\kappa(x)$ is the residue field of $X$ at $x$.
We say that $D$ is \emph{$f$-nef} if $D \cdot E_i \ge 0$ for all $i$.

\begin{defn}
A resolution $f: Y \to X$ of a normal surface singularity $(x \in X)$ is called \emph{minimal} if a canonical divisor $K_Y$ is $f$-nef.
\end{defn}

\begin{rem}
\
\begin{enumerate}[label=$($\roman*$)$]
\item A minimal resolution of a normal surface singularity exists by \cite[Theorem 2.25]{Kol}.
\item Let $f: Y \to X$ and $f': Y' \to X$ be minimal resolutions and $\mu: W \to X$ be a resolution which factors through both $Y$ and $Y'$.
Then it follows from \cite[Theorem 3.52 (2)]{KM}, combining with the rigidity lemma (c.f. \cite[Lemma 1.15]{Deb}), that $Y$ is isomorphic to $Y'$, that is, a minimal resolution is unique up to isomorphism.
\end{enumerate}
\end{rem}

In this paper, a \emph{graph} is a $\Z^3$-weighted undirected multigraph.

\begin{defn}
Let $(x \in X)$ be a normal surface singularity, $f: Y \to X$ be the minimal resolution with the exceptional divisor $\Exc(f) = \sum_{i=1}^n E_i$.
\begin{enumerate}
\item A \emph{dual graph} of $(x \in X)$ is a graph whose set of vertexes is $\{E_1, \dots, E_n\}$, the number of edges between $E_i$ and $E_j$ is $(E_i \cdot E_j) \in \N$ and the weight at $E_i$ is $(\dim_{\kappa(x)} H^0(E_i, \sO_{E_i}), g(E_i), (E_i^2)) \in \Z^3$.
\item We define the \emph{parameter} $r$ of the dual graph of $(x \in X)$ by
\[
r : = \min_{i} \dim_{\kappa(x)}(H^0(E_i, \sO_{E_i})) \in \N_{\ge 1}.
\]
\end{enumerate}
\end{defn}

\begin{rem}\label{rmk on dual graph}
With the above notation, the intersection numbers $(E_i^2)$ and $(E_i \cdot E_j)$ are divisible by $r_i : = \dim_{\kappa(x)}(H^0(E_i, \sO_{E_i}))$.
Moreover, since $f$ is minimal, we have the inequality $(E_i ^2) \le -2 r_i$.
\end{rem}

\begin{notation2}\label{notation}
Let $f: Y \to X$ be the minimal resolution of a normal surface singularity $(x \in X)$ with the exceptional divisor $\Exc(f) = \sum_{i=1}^n E_i$.
We draw the dual graph of $(x \in X)$ as follows:
\begin{itemize}
\item For each $i$, we represent the vertex $E_i$ with a circle:
\[
\begin{tikzpicture}
\node[draw, shape=circle, inner sep=4.5pt] at (0,0){};
\end{tikzpicture}
\]
\item Within each circle corresponding to $E_i$, we write the number $-a_i : = (E_i^2)$, the number $g_i: = g(E_i)$ below the circle, and the number $r_i: = \dim_{\kappa(x)}(H^0(E_i , \sO_{E_i}))$ above the circle. 
\[
\begin{tikzpicture}
\node[draw, shape=circle, inner sep=0.3pt] at (0,0){{\tiny $-a_i$}};
\draw[shift={(0,0.4)}] (0,0) node{\tiny $r_i$};
\draw[shift={(0,-0.4)}] (0,0) node{\tiny $g_i$};
\end{tikzpicture}
\]
If one has $g_i=0$ (resp.~$r_i=1$, $a_i=-2r_i$), then we omit writing $g_i$ (resp.~$r_i, a_i$).

\item If the intersection number $(E_i^2)$ is not determined, we denote it by
\begin{tikzpicture}
\node[draw, shape=circle, inner sep=1.8pt] at (0,0){{$*$}};
\end{tikzpicture}.
In other words, the symbol \dq{$*$} can take any integer divided by $r_i$ and smaller than or equal to $-2r_i$ (c.f. Remark \ref{rmk on dual graph}).
\item For each $i \neq j$, we draw $e: = (E_i \cdot E_j)$-parallel lines between the circles corresponding to $E_i$ and $E_j$.

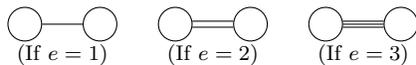
\vspace*{0.2cm}
\[
\begin{tikzpicture}
\node[draw, shape=circle, inner sep=4.5pt] (A) at (0,0){};
\node[draw, shape=circle, inner sep=4.5pt] (B) at (1,0){};
\draw (A) --(B);
\draw (0.5,-0.4) node{\tiny (If $e=1$)};

\node[draw, shape=circle, inner sep=4.5pt] (C) at (2,0){};
\node[draw, shape=circle, inner sep=4.5pt] (D) at (3,0){};
\draw (C.10)--(D.170);
\draw (C.-10)--(D.190);
\draw[shift={(2,0)}] (0.5,-0.4) node{\tiny (If $e=2$)};

\node[draw, shape=circle, inner sep=4.5pt] (E) at (4,0){};
\node[draw, shape=circle, inner sep=4.5pt] (F) at (5,0){};
\draw (E) --(F);
\draw (E.10)--(F.170);
\draw (E.-10)--(F.190);
\draw[shift={(4,0)}] (0.5,-0.4) node{\tiny (If $e=3$)};
\end{tikzpicture}
\]
We also draw as 
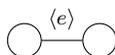
\vspace*{0.2cm}
\[
\begin{tikzpicture}
\node[draw, shape=circle, inner sep=4.5pt] (A) at (0,0){};
\node[draw, shape=circle, inner sep=4.5pt] (B) at (1,0){};
\draw (A) --node[auto=left]{\tiny $\langle e \rangle$}(B);
\end{tikzpicture}
\]
if $e \ge 4$ or if $e$ is not an explicit number.
\end{itemize}
\end{notation2}

\section{Lemmata on curves}

In this section, we prove some results about a curve over a (not necessarily algebraically closed) field with arithmetic genus zero or one, which we will need in Section \ref{4}.
Some of the results may be known to experts (even if it is not defined over an algebraically closed
field), but we include their proofs here for convenience.

\subsection{Smoothness of geometrically integral curves}

\begin{lem}\label{genus zero}
Let $C$ be a curve over a field $k$ with $g(C)=0$.
We assume that one of the following holds:
\begin{enumerate}[label=\textup{(\roman*)}]
\item $C$ is geometrically integral over $k$, or 
\item $C$ is Gorenstein, $K : = H^0(C, \sO_C)$ is separable over $k$ and there exists an invertible sheaf $L$ on $C$ such that the integer $\deg_{C/k}(L)/ \dim_k(K) $ is an odd number.
\end{enumerate}
Then $C$ is smooth over $k$.
\end{lem}

\begin{proof}
We first assume (ii).
Noting that $C$ has a $K$-scheme structure and $\deg_{C/K}(L) = \deg_{C/k}(L)/\dim_k (K)$, it follows from \cite[Lemma 10.6 (2), (3)]{Kol} that $C$ is isomorphic to $\PP^1_{K}$, which is smooth over $K$.
Combining this with the separability of $K$ over $k$, the assertion holds.

We next assume (i).
Let $C' : = C \times_{\Spec (k)} \Spec (\overline{k})$ be the base change of $C$ to the algebraic closure $\overline{k}$ of $k$.
By the flat base change theorem, we have $H^1(C', \sO_{C'})=0$.
Therefore, $C'$ is a curve over $\overline{k}$ with arithmetic genus zero, that is, $C'$ is isomorphic to $\PP^1_{\overline{k}}$.
Since $C'=C \times_k \overline{k}$ is smooth over $\overline{k}$, it follows from \cite[Lemma 02V4]{Sta} that $C$ is smooth over $k$.
\end{proof}


We next consider the singularities of curves with arithmetic genus one.
It follows from Tate's genus change formula (\cite{Tate}, cf.~\cite{Schr}, see also \cite{PW} and \cite{JW}) that a regular geometrically integral curve $C$ of genus one is smooth if the characteristic is larger than $3$.
In the next proposition, we also discuss the case where C is not regular.

\begin{prop}\label{genus one}
Let $C$ be a geometrically integral curve over a field $k$ with $g(C)=1$ and the characteristic $\ch(k)$ of $k$ larger than $3$.
Then one of the following holds:
\begin{enumerate}[label=$(\arabic*)$]
\item $C$ is smooth over $k$, or 
\item there exists a non-regular point $Q \in C$ such that $\kappa(Q) \cong k$ and $C \setminus \{Q\}$ is smooth over $k$.
\end{enumerate}
\end{prop}

\begin{proof}
    By Theorem \ref{smoothness of genus one} below, we may assume that $C$ is not regular.
    Let $\nu : \overline{C} \to C$ be the normalization.
    Since $\nu$ is birational and $C$ is geometrically integral, $\overline{C}$ is geometrically irreducible and geometrically $(R_0)$ over $k$.
    This shows that $\overline{C}$ is geometrically integral over $k$, which implies that $H^0(\overline{C}, \sO_{\overline{C}}) = k$.
    Therefore, the short exact sequence 
    \[
    0 \to \sO_C \to \nu_* \sO_{\overline{C}} \to \mathcal{F}:=\nu_*\sO_{\overline{C}}/\sO_C \to 0
    \]
    induces the exact sequence
    \[
    0 \to H^0(C, \mathcal{F}) \to H^1(C, \sO_{C}) \to H^1(\overline{C}, \sO_{\overline{C}}) \to 0.
    \]
    
    Since $H^0(C, \mathcal{F}) \neq 0$ and $\dim_k H^1(C, \sO_C)=1$, we conclude that 
    \[
    \left\{\begin{matrix} g(\overline{C})=0 , \textup{ and}\\ \dim_k H^0(C, \mathcal{F})=1. \end{matrix}\right. 
    \]
    The former statement implies that $\overline{C}$ is smooth (Lemma \ref{genus zero} (1)).
    By the latter condition, $\nu$ is an isomorphism outside a single $k$-rational point, which completes the proof.
\end{proof}

\subsection{Geometric reducedness of curves}

In this subsection, we give sufficient conditions for curves with genus zero or one to be geometrically reduced.
We first consider the genus zero case.

\begin{lem}\label{separability genus zero}
Let $C$ be a curve over a field $k$ with $\ch(k) > 2$ and $g(C)=0$.
We further assume that $C$ is Gorenstein and the field $H^0(C,\sO_C)$ is separable over $k$.
Then $C$ is geometrically reduced over $k$.
\end{lem}

\begin{proof}
After replacing $k$ by $H^0(C,\sO_C)$, we may assume that $H^0(C, \sO_C)=k$.
By \cite[Lemma 10.6 (3)]{Kol}, $C$ is isomorphic to a conic in $\PP^2_k$.
If $C$ is not geometrically reduced, then there exists a purely inseparable finite field extension $\ell$ of $k$ such that the pullback $\mu^*C$ of the divisor $C$ by the natural morphism $\mu: \PP^2_{\ell} \to \PP^2_k$ has a coefficient divided by $p:= \ch(k)$.
Since $p>2$, this is a contradiction to the fact that $\mu^* C$ is linearly equivalent to $\sO_{\PP^2_{\ell}}(2)$.
\end{proof}

We next consider the genus one case.

\begin{thm}[\textup{\cite[Theorem 1.1]{PW}, cf.~\cite[Proposition 9.11 (2)]{Tan}}]\label{smoothness of genus one}
Let $C$ be a regular curve over a field $k$ with $\ch(k) > 3$ and $g(C)=1$.
We further assume that the field $H^0(C,\sO_C)$ is separable over $k$.
Then $C$ is smooth over $k$.
\end{thm}

\begin{defn}\label{defn: node}
An excellent local ring $(A,\m)$ is a \emph{node} if there exists an isomorphism 
\[
\widehat{A} \cong R/(f)
\]
of rings where $(R,\n)$ is a regular local ring of dimension $2$ and $f \in \n^2$ is an element such that $\overline{u f} \in R/\n^3$ is not a square in the ring $R/\n^3$ for every unit $u \in R$.
\end{defn}

\begin{rem}\label{rem: node}
    The above definition differs slightly from that in \cite[Paragraph 1.41]{Kol}. 
    Below, we provide some clarifications regarding this distinction:
    \begin{enumerate}[label=\textup{(\arabic*)}]
        \item The two definitions are genuinely distinct when considered over a non algebraically closed field. 
        For instance, $A = \R[[x,y]] / (-x^2 + y^3)$ is a node under the definition in \cite[Paragraph 1.41]{Kol}, but not under the present definition.
        \item Nevertheless, certain properties established in \cite{Kol} remain valid for our definition. 
        For example, as stated in \cite[Theorem 2.31]{Kol}, if $(x \in X)$ is a normal excellent surface singularity and $(X, B)$ is log canonical with $B$ a reduced divisor, then $B$ is either regular or a node (in the sense of Definition \ref{defn: node}). 
        The proof is analogous but relies on Lemma \ref{equiv on node} below, rather than \cite[Claim 1.41.1]{Kol}.
    \end{enumerate}
\end{rem}

\begin{lem}\label{equiv on node}
    Let $(A, \m, \kappa = A/\m)$ be an excellent reduced local ring of dimension one, $\overline{A}$ the integral closure of $A$ in its total ring of fractions, and $J \subseteq \overline{A}$ the intersection of all maximal ideals of $\overline{A}$. 
    Then the following conditions are equivalent:
    \begin{enumerate}[label=\textup{(\arabic*)}]
        \item $A$ is a node (in the sense of Definition \ref{defn: node}).
        \item $J \subseteq A$ and $\dim_{\kappa} (\overline{A}/J) = 2$.
    \end{enumerate}
\end{lem}

\begin{proof}
    Since $A$ is excellent, $\overline{A}$ is finite over $A$ and $\widehat{A}$ is reduced. 
    Furthermore, the base change $R := \overline{A} \otimes_A \widehat{A}$ is the normalization of $\widehat{A}$. 
    The natural homomorphism 
    \[
    \overline{A}/\m\overline{A} \to R/\m R\] is an isomorphism, as it is the base change of the isomorphism $A/\m \xrightarrow{\sim} \widehat{A}/\m\widehat{A}$.
    Consequently, the quotient ring $R/JR$ is isomorphic to $\overline{A}/J$, which is a finite product of fields. 
    Combining this with the fact that every maximal ideal of $R$ lies over a maximal ideal of $\overline{A}$, we deduce that $JR$ is the intersection of all maximal ideals of $R$. 
    By replacing $A$ with $\widehat{A}$, we may assume that $A$ is complete. 
    In particular, by the Cohen Structure Theorem, $A$ is a quotient of a regular local ring.

    The implication (2) $\Rightarrow$ (1) follows similarly to the proof of \cite[Claim 1.41.1]{Kol}
    (observe that $J$ is a principal divisor, as $\overline{A}$ is a finite product of DVRs by \cite[Proposition 4.3.2]{HS}, which implies that the associated graded ring $\mathrm{Gr}_J(\overline{A})$ is isomorphic to $(A/J)[t]$).

    For the converse implication (1) $\Rightarrow$ (2), since the order of $f \in R$ is two, if $x, y$ are sufficiently general generators of the maximal ideal $\n \subseteq R$, we can express $f$ as 
    \begin{align}\label{eq:node}
        f = u x^2 + v xy + w y^2,
    \end{align}
    where $u, v, w \in R$ and $u$ is a unit. 
    Moreover, the polynomial $\overline{u}T^2 + \overline{v}T + \overline{w} \in \kappa[T]$ is irreducible or has two distinct roots in $\kappa$, as $\overline{u^{-1}f} \in R/\n^3$ is not a square.

    Equation \eqref{eq:node} implies that if $y$ is a zero divisor in $A$, then $x$ is also a zero divisor, contradicting the assumption that $A$ is reduced and one-dimensional. Thus, $x/y$ belongs to the total ring of fractions of $A$. By the dimension formula (\cite[Theorem B.5.1]{HS}), every maximal ideal of $B := A[x/y]$ has height one. 
    Since we have
    \[
        B \cong A[T]/(uT^2 + vT + w, yT - x),
    \]
    the quotient ring $B/\m B$ is either isomorphic to $\kappa \times \kappa$ or a degree-two field extension of $\kappa$. As $\m B$ is a principal ideal $(y)$, the ring $B$ is integrally closed. 
    Thus, we have $B=\overline{A}$ and $J = (y)$, which implies (2).
\end{proof}

\begin{lem}\label{node2}
Let $(A, \m_A, \kappa)$ be a Noetherian local domain which is a node and $a \in \Frac (A)$ be an element in the fraction field $\Frac (A)$ of $A$.
We assume that the characteristic $p : = \ch(\kappa)$ of $\kappa$ is larger than $2$.
If we have $a^p \in A$ then one has $a \in A$.
\end{lem}

\begin{proof}
We first note that the following diagram 
\[
\xymatrix{
A \ar[r] \ar[d] & \overline{A} \ar[d] \\
\kappa \ar[r] & \overline{A}/(\m_A \cdot \overline{A})
}
\]
is a Cartesian diagram of $A$-modules since the kernels of vertical maps are isomorphic.
By Lemma \ref{equiv on node}, the quotient ring $\overline{A}/(\m_A \cdot \overline{A})$ is isomorphic to $\kappa \times \kappa$ or a field extension $\ell$ of $\kappa$ with $[\ell: \kappa]=2$.
We note that in the latter case, $\ell$ is separable over $\kappa$ since we have $p \neq 2$.

In the first case, it follows from the above Cartesian diagram that 
\[
A \cong \{ x \in \overline{A} \mid p_1(x) = p_2(x)  \in \kappa \},
\]
where $p_i : \overline{A} \to \kappa$ is the projection to the $i$-th component.
Since we have $a^p \in A$, one has $p_1(a^p)=p_2(a^p) \in \kappa$.
Therefore, we have $p_1(a)^p = p_2(a)^p \in \kappa.$
This proves that $p_1(a)=p_2(a) \in \kappa$, and hence $a \in A$.

In the latter case, we have
\[
A = \{ x \in \overline{A} \mid \textup{the image $\overline{x} \in \ell =\overline{A}/(\m_A \cdot \overline{A})$ of $x$ is contained in the subfield $\kappa$}\}.
\]
Since we have $a^p \in A$, the element $\overline{a^p} \in \ell$ is contained in $\kappa$.
Combining this with the separability of the extension $\ell \supseteq \kappa$, we conclude that $\overline{a} \in \kappa$, which proves the assertion.
\end{proof}

\begin{prop}\label{separability}
Let $C$ be a curve over a field $k$ with $\ch(k)>2$ and $H^0(C, \sO_C) = k$.
Assume that $\sO_{C,P}$ is regular or nodal for every closed point $P \in C$.
Then there exists no inseparable finite extension of $k$ contained in the function field $K(C)$.
\end{prop}

\begin{proof}
Assume that the assertion is not true.
Then there exists an element $a \in K(C)$ such that $a \in k^{1/p} \setminus k$.
Since $C$ is integral, we may consider $\sO_{C,P}$ as a subrings of $K(C)$ for every closed point $P$.
Then it follows from the sheaf condition for $\sO_C$ that we have the equation
\[
H^0(C, \sO_{C}) = \bigcap_{P \in C} \sO_{C,P}.
\]
Therefore, there exists a closed point $P \in C$ such that $a \not\in \sO_{C,P}$ and $a^p \in \sO_{C,P}$.
If $P$ is a regular point, then this is a contradiction because $\sO_{C,P}$ is integrally closed.
Otherwise, it is also a contradiction by Lemma \ref{node2}.

\end{proof}

\begin{cor}\label{separability genus one}
Let $C$ be a curve over a field $k$ with $H^0(C, \sO_C)=k$, $g(C)=1$ and $\ch(k) >2$.
We further assume that $C$ is not regular, but every non-regular point is a node.
Then $C$ is geometrically reduced over $k$.
\end{cor}

\begin{proof}
Let $\nu:\overline{C} \to C$ be the normalization of $C$.
We set $\ell : = H^0(\overline{C}, \sO_{\overline{C}})$ and $\mathcal{F} : = (\nu_*\sO_{\overline{C}}) / \sO_C$.
By taking the global section of the short exact sequence
\[
0 \to \sO_C \to \nu_* \sO_{\overline{C}} \to \mathcal{F} \to 0,
\]
we obtain the exact sequence
\[
0 \to k \to \ell \to H^0(C, \mathcal{F}) \to k \to H^1(\overline{C}, \sO_{\overline{C}}) \to 0
\]
of $k$-modules.
By counting dimension, we have
\[
[\ell:k] = \dim_k H^0(C, \mathcal{F}) +[\ell:k] \dim_\ell (H^1(\overline{C}, \sO_{\overline{C}})).
\]
Since we have $H^0(C, \mathcal{F}) \neq 0$, we conclude that $H^1(\overline{C}, \sO_{\overline{C}})=0$.
On the other hand, $\ell$ is separable over $k$ by Proposition \ref{separability}.
Therefore, it follows from Lemma \ref{separability genus zero} that $\overline{C}$ is geometrically reduced over $k$, which implies the assertion in the corollary.
\end{proof}

\section{Geometrically log canonical (resp.~geometrically klt) singularities}\label{4}

In this section, we will give sufficient conditions for normal surface singularities to be geometrically log canonical (resp.~geometrically klt).

\begin{defn}
Let $(x \in X)$ be a normal surface singularity, $f: Y \to X$ be a proper birational morphism from an integral scheme $Y$ and $\Exc(f) = \sum_{i=1}^n E_i$ be the irreducible decomposition of the exceptional locus.
\begin{enumerate}
\item We say that $f$ satisfies the \emph{$(*)$-condition} if 
\begin{enumerate}
\item $E_i$ is a smooth curve over $\kappa(x)$ for every $i$, and 
\item the scheme theoretic intersection $E_i \cap E_j$ is smooth over $\kappa(x)$ for every $i \neq j$.
\end{enumerate}
\item We say that $f$ satisfies the \emph{$(**)$-condition} if there exists an \'{e}tale surjective morphism $\phi : X' \to X$ from a scheme $X'$ such that the base change $f'_{x'}: Y \times_X \Spec(\sO_{X',x'}) \to \Spec(\sO_{X',x'})$ of $f$ satisfies the $(*)$-condition for every closed point $x' \in X'$.
\end{enumerate}
\end{defn}

\begin{lem}\label{criterion2}
Let $X$ be a $2$-dimensional normal variety over a field $k$.
We assume that the following conditions are satisfied:
\begin{enumerate}[label=\textup{(\roman*)}]
\item $X$ is smooth over $k$ outside finitely many closed points $x_1, \dots, x_n \in X$.
\item For every $i$, the residue field $\kappa(x_i)$ of $X$ at $x_i$ is separable over $k$.
\item For every $i$, there exists a resolution $f_i : Y_i \to \Spec(\sO_{X,{x_i}})$ which  satisfies the $(**)$-condition.
\end{enumerate}
If $X$ is log canonical (resp.~klt), then $X$ is geometrically log canonical (resp.~geometrically klt) over $k$.
\end{lem}

\begin{proof}
After shrinking $X$, we may assume that $X$ is smooth over $k$ outside a single closed point $x \in X$.
Moreover, we may assume that there exists a resolution $f: Y \to X$ which is isomorphic outside $x$ and whose restriction to $\Spec(\sO_{X,x})$ satisfies the $(**)$-condition.
We note that $X$ is not regular at $x$ since $\kappa(x)$ is separable over $k$ (\cite[Lemma 00TV]{Sta}).
In particular, we have $\Exc(f) \neq \emptyset$.

It then follows from the definition of $(**)$-condition that there exist an \'{e}tale morphism $\phi: X' \to X$ from a normal variety $X'$ such that $\phi^{-1}(x)$ is non-empty and the base change $f' : Y' : = Y \times_{X} X' \to X'$ of $f$ satisfies the $(*)$-condition after restricting to $\Spec(\sO_{X',x'})$ for every point $x' \in \phi^{-1}(x)$.
\[
\xymatrix{
Y' \ar^-{\psi}[r] \ar^-{f'}[d] & Y \ar^-{f}[d] \\
X' \ar^-{\phi}[r] & X
}\]

Since the morphism $\psi : Y' \to Y$ is an \'{e}tale morphism, the morphism $f'$ is again a resolution of singularities and every coefficient of $K_{Y'/X'} = \psi^*(K_{Y/X})$ (c.f. \cite[Lemma 2.6]{ST mixed}) is at least (resp.~larger than) $-1$.
On the other hand, since $\kappa(x')$ is separable over $k$ and $f'$ satisfies the $(*)$-condition after restricting to $\Spec(\sO_{X',x'})$, we conclude that the pair $(Y', \Exc(f'))$ is geometrically SNC over $k$.
Therefore, it follows from Lemma \ref{criterion1} that $X'$ is geometrically log canonical (resp.~geometrically klt) over $k$, that is, $X' \times_k \ell$ is log canonical (resp.~klt) for every finite extension $\ell$ of $k$.
Then the assertion in the lemma follows by applying \cite[Lemma 2.6]{ST mixed} to the \'{e}tale morphism $\phi \times_k \ell : X' \times_k \ell \to X \times_k \ell$.
\end{proof}

\subsection{Rational singularities}

In this subsection, we give a sufficient condition for $2$-dimensional log canonical and rational singularities to satisfy the $(**)$-condition.

\begin{lem}\label{residue extension}
Let $(x \in X)$ be a normal surface singularity and $\ell$ be a finite separable field extension of the residue field $\kappa : = \kappa(x)$.
Then there exists an \'{e}tale surjective morphism $\phi: X' \to X$ from a normal surface singularity $(x' \in X')$ such that the residue field $\kappa(x')$ is isomorphic to $\ell$.
\end{lem}

\begin{proof}
Let $R$ be the structure ring of $X$ and $\m$ be the maximal ideal of $R$.
Since $\ell$ is finite separable extension of $\kappa$, there exists a monic polynomial $g(t) \in \kappa[t]$ such that $\ell \cong \kappa[t]/(g)$.
Take a monic polynomial $G(t) \in R[t]$ whose image to $\kappa[t]$ corresponds to $g(t)$ and 
we set $B : = R[t]/(G)$.
Then the natural morphism $\phi : X' :  = \Spec (B) \to \Spec (R)$ is a desired morphism.
\end{proof}

\begin{prop}\label{* for rational}
Let $(x \in X)$ be a normal surface singularity.
We assume that the following conditions are satisfied:
\begin{enumerate}[label=\textup{(\roman*)}]
\item the characteristic $\ch(\kappa(x))$ of the residue field $\kappa(x)$ satisfies $\ch(\kappa(x)) > 3$, and
\item $(x \in X)$ is a log canonical and rational singularity,
\end{enumerate}
Then the minimal resolution $f: Y \to X$ satisfies the $(**)$-condition.
\end{prop}

\begin{proof}
Let $\Exc(f)=\sum_{i=1}^m E_i$ be the irreducible decomposition of the exceptional locus $\Exc(f)$.
We consider $E_i$ as the reduced closed subscheme of $Y$.

Take a separable finite extension $\ell$ of $\kappa(x)$ such that every irreducible component of the base change $E_i \times_{\kappa(x)} \ell$ is geometrically irreducible over $\ell$  for every $i$.
By Lemma \ref{residue extension}, there exists an \'{e}tale surjective morphism $\phi :X' \to X$ from a normal surface singularity $(x' \in X')$ with the residue field $\ell$.
We note that the base change 
\[
f' : Y' : = Y \times_X X' \to X'
\]
of $f$ is again the minimal resolution since $K_{Y'/X'}$ is the pullback of $K_{Y/X}$ (c.f. \cite[Lemma 2.6]{ST mixed}).
Therefore, by replacing $X$ by $X'$, we may assume that every irreducible component of $\Exc(f)$ is geometrically irreducible.

On the other hand, It follows from the classification of the dual graph (Theorem \ref{classification} (1) and (2) below) that we have $g(E_i)=0$ and $r_i : = \dim_{\kappa(x)} (H^0(E_i, \sO_{E_i})) \le 4$ for all $i$.
In particular, each $E_i$ is smooth over $\kappa(x)$ by Lemma \ref{genus zero} (i) and Lemma \ref{separability genus zero}.

We next show that the scheme theoretic intersection $E_i \cap E_j$ is smooth over $\kappa(x)$ for every $i \neq j$.
It follows from the classification that $(E_i \cdot E_j)$ is either $0$ or $\max\{r_i, r_j\}$.
We may assume that $(E_i \cdot E_j)=r_i$.
Combining the equation
\[
r_i = (E_i \cdot E_j) = \sum_{P \in E_i \cap E_j} \dim_{\kappa(x)} (\sO_{E_i \cap E_j, P})
\]
with the fact that $\sO_{E_i \cap E_j, P}$ is an algebra over the field $K_i : = H^0(E_i, \sO_{E_i})$ whose degree over $\kappa(x)$ is $r_i$, the scheme theoretic intersection $E_i \cap E_j$ is isomorphic to $\Spec(K_i)$.
Since the extension degree $r_i$ of $K_i$ over $\kappa(x)$ is not divisible by $\ch(\kappa(x))$, the scheme theoretic intersection $E_i \cap E_j$ is smooth over $\kappa(x)$.
\end{proof}

\begin{cor}\label{geometric klt}
Let $X$ be a $2$-dimensional geometrically normal variety over a field $k$ of characteristic $p>3$ and $x_1, \dots, x_n \in X$ be the non-smooth points of $X$ over $k$.
Suppose that $X$ is klt at $x_i$ and the residue field $\kappa(x_i)$ of $X$ at $x_i$ is separable over $k$ for every $i$.
Then $X$ is geometrically klt over $k$.
\end{cor}

\begin{proof}
This follows from Lemma \ref{criterion2} and Proposition \ref{* for rational}.
\end{proof}

\begin{eg}\label{counterexample}
In Corollary \ref{geometric klt}, the assumption that $p$ is larger than $3$ is optimal.
We give counter-examples in characteristic $2$ and $3$.
\begin{enumerate}[label=\textup{(\roman*)}]
\item Let $k$ be a field of characteristic $2$ and $a \in k$ be an element with $\sqrt{a} \not\in k$.
Then the normal surface
\[
X : = \Spec(k[x,y,z]/(z^2+x^3+ay^4+y^7))
\] is smooth over $k$ outside the origin $P : =(0,0,0) \in X$ and is klt at $x$ because the dual graph of $(P \in X)$ is
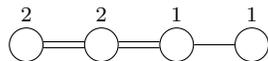
\vspace*{0.2cm}
\[
\begin{tikzpicture}
\node[draw, shape=circle, inner sep=4.5pt] (A0) at (-1,0){};
\draw[shift={(0,0.4)}] (-1,0) node{\tiny $2$};
\node[draw, shape=circle, inner sep=4.5pt] (A1) at (0,0){};
\draw[shift={(0,0.4)}] (0,0) node{\tiny $2$};
\node[draw, shape=circle, inner sep=4.5pt] (A2) at (1,0){};
\draw[shift={(0,0.4)}] (1,0) node{\tiny $1$};
\node[draw, shape=circle, inner sep=4.5pt] (A3) at (2,0){};
\draw[shift={(0,0.4)}] (2,0) node{\tiny $1$};

\draw (A0.10)--(A1.170);
\draw (A0.-10)--(A1.190);
\draw (A1.10)--(A2.170);
\draw (A1.-10)--(A2.190);
\draw (A2)--(A3);
\end{tikzpicture}
\]
which appears in the classification (Figure \ref{classification klt}) of numerically klt graphs.
However, by considering the change of coordinate $z \mapsto z+\sqrt{a}y^2$, the base change $X_\ell : = X \times_k \ell$ of $X$ to $\ell : = k(\sqrt{a})$ is isomorphic to $\Spec(\ell[x,y,z]/(z^2+x^3+y^7))$, which is not klt, and more strongly, not log canonical.
\
\item Let $k$ be a field of characteristic $3$ and $a \in k$ be an element with $ \sqrt[3]{a} \not\in k$.
Then the normal surface 
\[
X : = \Spec(k[x,y,z]/(z^2+x^3-ay^3+y^7))
\] is smooth over $k$ outside the origin $P : =(0,0,0) \in X$ and is klt at $x$ because the dual graph of $(P \in X)$ is
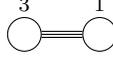
\[
\begin{tikzpicture}
\node[draw, shape=circle, inner sep=4.5pt] (B1) at (0,0){};
\draw[shift={(0,0.4)}] (0,0) node{\tiny $3$};
\node[draw, shape=circle, inner sep=4.5pt] (B2) at (1,0){};
\draw[shift={(0,0.4)}] (1,0) node{\tiny $1$};

\draw (B1.10)--(B2.170);
\draw (B1)--(B2);
\draw (B1.-10)--(B2.190);
\end{tikzpicture}
\]
which appears in the classification (Figure \ref{classification klt}) of numerically klt graphs.
However, by considering the change of coordinate $x \mapsto x+\sqrt[3]{a}y$, the base change $X_\ell : = X \times_k \ell$ of $X$ to $\ell : = k(\sqrt[3]{a})$ is isomorphic to $\Spec(\ell[x,y,z]/(z^2+x^3+y^7))$, which is not log canonical.

\end{enumerate}
\end{eg}

\subsection{(Twisted) cusp singularities}

In this subsection, we give a sufficient condition for (twisted) cusp singularities to satisfy the $(**)$-condition.

\begin{lem}\label{H^0-lemma}
Let $Y$ be a Noetherian integral scheme with $\dim Y = 2$ and $Z, E$ be reduced Weil divisors on $Y$ with no common component.
We write $Z \cap E = \{P_1, \dots, P_n\}$.
We assume that the pair $(Y,Z+E) $ is SNC around $P_i$ for every $i$.
Then the ring $H^0(Z+E, \sO_{Z+E})$ is naturally isomorphic to the subring
\[
\{ (s,t) \in H^0(Z, \sO_Z) \times H^0(E, \sO_E) \mid s(P_i) = t(P_i)  \in \kappa(P_i), \ \forall i\}
\]
of $H^0(Z, \sO_Z) \times H^0(E, \sO_E)$, where $s(P_i)$ (resp.~$t(P_i)$) is the image of $s$ (resp.~$t$) by the natural morphism from $H^0(Z, \sO_Z)$ (resp.~$H^0(E, \sO_E)$) to the residue field $\kappa(P_i)$.
\end{lem}

\begin{proof}
Let $f_i$ and $g_i$ be the natural surjections
\[
f_i : \sO_{Z+E,P_i} \onto \sO_{Z, P_i} \ \ \textup{and } g_i: \sO_{Z+E, P_i} \onto \sO_{E, P_i}
\]
for every $i$.
We first show that the following diagram
\begin{equation}\label{cartesian1}
\xymatrix{
H^0(Z+E, \sO_{Z+E}) \ar[r] \ar[d] & H^0(Z, \sO_Z) \times H^0(E, \sO_E) \ar[d] \\
\prod_{i=1}^n \sO_{Z+E, P_{i}} \ar^-{\prod_i (f_i \times g_i)}[r] & \prod_{i=1}^n ( \sO_{Z,P_i} \times \sO_{E, P_i})
}
\end{equation}
is Cartesian, in other words, $H^0(Z+E, \sO_{Z+E})$ is naturally isomorphic to the subring
\[
A: = \{(s,t, \alpha_1, \dots, \alpha_n) \in B \mid \forall i, \ s_{P_i}=f_i(\alpha_i) \textup{ and } t_{P_i}=g_i(\alpha_i)\}
\]
of the ring 
\[
B : = H^0(Z,\sO_Z) \times H^0(E, \sO_E) \times \prod_i \sO_{Z+E, P_i} ,\] 
where $s_{P_i}$ (resp.~$t_{P_i}$) is the stalk of $s$ (resp.~$t$) at $P_i$.
Let $\phi: H^0(Z+E, \sO_{Z+E}) \to B$ be the natural morphism.
Since $\phi$ is obviously injective, it is enough to show that $\mathrm{Im}(\phi)=A$.
Take an element $(s,t,\alpha_1, \dots, \alpha_n) \in A$.
Let $Z^{\circ}$ and $B^{\circ}$ be the open subschemes of $Z+E$ defined by
\begin{eqnarray*}
Z^{\circ}  : = Z \setminus \{P_1, \dots, P_n\}, &\textup{ and}\\
E^{\circ}  : = E \setminus \{P_1, \dots, P_n\}. &
\end{eqnarray*}
We write $s^{\circ}$ (resp.~$t^{\circ}$) as the restriction of $s$ (resp.~$t$) to $H^0(Z^{\circ}, \sO_Z) \cong H^0(Z^{\circ}, \sO_{Z+E})$ (resp.~$ H^0(E^{\circ}, \sO_E) \cong H^0(E^{\circ}, \sO_{Z+E})$).
For every $i$, take an open neighborhood $U_i \subseteq Z+E$ of $P_i$ and a section $u_i \in H^0(U_i, \sO_{Z+E})$ whose stalk at $P_i$ coincides with $\alpha_i$.
Since we have $s_{P_i} = f_i(\alpha_i) = (u_i|_{Z})_{P_i}$, after shrinking $U_i$, the restrictions of $s$ and $u_i$ to $U_i \cap Z$ coincide.
Therefore, we have the equation
\[
s^{\circ}|_{U_i \cap Z^{\circ}} = u_i|_{U_i \cap Z^{\circ}} \in H^0(U_i \cap Z^{\circ}, \sO_{Z+E}).
\]
Similarly, after shrinking $U_i$, we have
\[
t^{\circ}|_{U_i \cap E^{\circ}} = u_i|_{U_i \cap E^{\circ}} \in H^0(U_i \cap E^{\circ}, \sO_{Z+E})
\]
for every $i$.
Since $Z+E = Z^{\circ} \cup E^{\circ} \cup U_1 \cup \dots \cup U_n$ is an open immersion, the local sections $s^{\circ}, t^{\circ}, u_1, \dots, u_n$ patch together to give a global section $u \in H^0(Z+E, \sO_{Z+E})$ as desired.

We next show that for every $i$, the following diagram 
\begin{equation}\label{cartesian2}
\xymatrix{
\sO_{Z+E, P_i} \ar^-{f_i \times g_i }[r] \ar[d] & \sO_{Z,P_i} \times \sO_{E,P_i} \ar[d] \\
\kappa(P_i) \ar^-{\Delta}[r] & \kappa(P_i) \times \kappa(P_i)
}
\end{equation}
is Cartesian, where the vertical maps are the natural surjections and $\Delta: \kappa(P_i) \to \kappa(P_i) \times \kappa(P_i)$ is the diagonal map.
By the assumption, $R : = \sO_{Y, P_i}$ is a two-dimensional regular local ring and if $x \in R$ and $y \in R$ are defining equations of $Z$ and $E$ at $P_i$, respectively, then $x, y$ is a regular system of parameter of $R$. 
Since we have 
\[
\sO_{Z+E, P_i} =R/(xy),\  \sO_{Z,P_i} =R/(x) \textup{ and } \sO_{E, P_i} = R/(y),\] 
it is straightforward to verify that the diagram \eqref{cartesian2} is Cartesian.

By combining the Cartesian diagrams \eqref{cartesian1} and \eqref{cartesian2}, we obtain the Cartesian diagram 
\[
\xymatrix{
H^0(Z+E, \sO_{Z+E}) \ar[r] \ar[d] & H^0(Z, \sO_Z) \times H^0(E, \sO_E) \ar[d] \\
\prod_{i=1}^n \kappa(P_i) \ar^-{\prod_{i=1}^n \Delta}[r] & \prod_{i=1}^n ( \kappa(P_i) \times \kappa(P_i))
},
\]
which completes the proof.
\end{proof}

\begin{prop}\label{* for cusp}
Let $(x \in X)$ be a normal surface singularity.
We assume that the following conditions are satisfied:
\begin{enumerate}[label=\textup{(\roman*)}]
\item $(x \in X)$ is log canonical, and
\item the dual graph of $(x \in X)$ is a cusp with parameter $r \ge 1$ and length $n \ge 3$, that is, the number of vertices is $n \ge 3$ and the shape is as follows (c.f. Notation \ref{notation}):

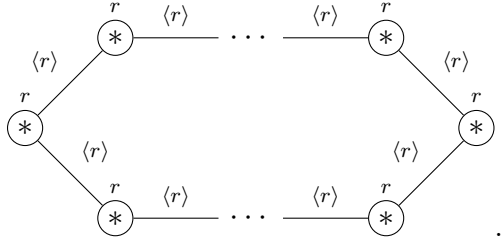
\vspace*{0.3cm}
\[
\begin{tikzpicture}
\node[draw, shape=circle, inner sep=1.8pt] (A) at (0,0){$*$};
\draw[shift={(0,0.4)}] (0,0) node{\tiny $r$};
\node[draw, shape=circle, inner sep=1.8pt] (B) at (1.2,1.2){$*$};
\draw[shift={(0,0.4)}] (1.2,1.2) node{\tiny $r$};
\draw node (C) at (3,1.2){$\cdots$};
\node[draw, shape=circle, inner sep=1.8pt] (D) at (4.8 ,1.2){$*$};
\draw[shift={(0,0.4)}] (4.8,1.2) node{\tiny $r$};
\node[draw, shape=circle, inner sep=1.8pt] (E) at (6,0){$*$};
\draw[shift={(0,0.4)}] (6,0) node{\tiny $r$};
\node[draw, shape=circle, inner sep=1.8pt] (F) at (4.8,-1.2){$*$};
\draw[shift={(0,0.4)}] (4.8,-1.2) node{\tiny $r$};
\draw node (G) at (3,-1.2){$\cdots$};
\node[draw, shape=circle, inner sep=1.8pt] (H) at (1.2,-1.2){$*$};
\draw[shift={(0,0.4)}] (1.2,-1.2) node{\tiny $r$};

\draw (A) --node[auto=left]{\tiny $\langle r \rangle$} (B);
\draw (B) --node[auto=left]{\tiny $\langle r \rangle$} (C);
\draw (C) --node[auto=left]{\tiny $\langle r \rangle$} (D);
\draw (D) --node[auto=left]{\tiny $\langle r \rangle$} (E);
\draw (A) --node[auto=left]{\tiny $\langle r \rangle$} (H);
\draw (H) --node[auto=left]{\tiny $\langle r \rangle$} (G);
\draw (G) --node[auto=left]{\tiny $\langle r \rangle$} (F);
\draw (F) --node[auto=left]{\tiny $\langle r \rangle$} (E);
\end{tikzpicture}.
\]
\end{enumerate}
Then $H^0(E_i, \sO_{E_i})$ is a cyclic extension of $\kappa(x)$ for every vertex $E_i$.
In particular, the minimal resolution $f: Y \to X$ satisfies the $(*)$-condition.
\end{prop}

\begin{proof}
Let $\Exc(f) = \sum_{i=1}^n E_i$ be the irreducible decomposition and we write $K_i : = H^0(E_i , \sO_{E_i})$.
After renumbering, we may assume that $E_i \cap E_{i+1}$ is non-empty for all $i=1,2, \dots n$, where we set $E_{n+1} : = E_1$.
We note that in the case where $n \ge 4$, it follows from the shape of the dual graph that $E_i \cap E_j \cap E_k$ is empty for every $i<j<k$.
The same holds even in the case where $n=3$, because $(x \in X)$ is log canonical.

On the other hand, for every $i$, it follows from the equation 
\[
(E_i \cdot E_{i+1})=[K_i:\kappa(x)] = [K_{i+1}:\kappa(x)]=r\] 
that the following properties hold:
\begin{enumerate}[label=\textup{(\arabic*)}]
\item $E_i$ and $E_{i+1}$ intersects transversally at a single point $P_i$,
\item the natural morphism $f_i : K_i =H^0(E_i,\sO_{E_i}) \to \kappa(P_i)$ is an isomorphism, and 
\item the natural morphism $g_i : K_{i+1} =H^0(E_{i+1},\sO_{E_{i+1}}) \to \kappa(P_i)$ is an isomorphism, where we set $K_{n+1} : = K_1$.
\end{enumerate}

We write 
\[
Z : =E_1+\cdots + E_{n-1} = E - E_n.\]
By repeatedly applying Lemma \ref{H^0-lemma}, the natural restriction morphism $h_i : H^0(Z, \sO_Z) \to H^0(E_i,\sO_{E_i}) = K_i$ is isomorphic for every $i=1,2, \dots, n-1$.
It again follows from Lemma \ref{H^0-lemma} that we have the isomorphism
\[
H^0(E,\sO_E) \cong \{(s, t) \in H^0(Z,\sO_Z) \times K_n \mid f_{n-1}(h_{n-1}(s)) = g_{n-1}(t) , \ g_n  (h_{1}(s)) = f_n(t) \}.
\]
as $\kappa(x)$-algebras.
Since $f_i, g_i, h_i$ are isomorphic, this $\kappa(x)$-algebra is isomorphic to the invariant subring $H^0(Z,\sO_Z)^\sigma$ of $H^0(Z, \sO_Z)$, where the isomorphism 
\[
  \sigma : H^0(Z,\sO_Z) \xrightarrow{\sim} H^0(Z,\sO_Z)
\] of $\kappa(x)$-algebras is the composite map
\[
\xymatrix{
  H^0(Z,\sO_Z) \ar^-{h_1}[d] \ar^-{\sigma}[rrrr] & & & & H^0(Z,\sO_Z).\\
 K_1 \ar_-{g_n}[rd] &  & K_n \ar_-{g_{n-1}}[rd] &  & K_{n-1} \ar^-{(h_{n-1})^{-1}}[u]   \\
 &\kappa(P_n) \ar^-{(f_n)^{-1}}[ru] & & \kappa(P_{n-1}) \ar^-{(f_{n-1})^{-1}}[ru] &  
}\]
Therefore, $H^0(Z, \sO_Z)$ is a cyclic extension of $H^0(E, \sO_E)$.
Since $H^0(E, \sO_E)$ is isomorphic to $\kappa(x)$ by \cite[Corollary 10.10]{Kol}, we obtain the first assertion.
We now apply Lemma \ref{genus zero} (ii) to obtain the second assertion.
\end{proof}

\begin{rem}\label{rmk on irr cusp}
With the notation above, we further assume that $r \neq 1$.
Since $K_i : =H^0(E_i, \sO_{E_i})$ is isomorphic to $K_1$ and $K_i$ is a Galois extension of degree $r$ over $\kappa(x)$, we have the isomorphism $K_i \otimes_{\kappa(x)} K_1 \cong K_1^{r}$.
Therefore, the following property hold.
\begin{enumerate}
\item $E_i \times_{\kappa(x)} K_1 \cong \coprod_{i=1}^r \PP^1_{K_1}$.
\item Let $\phi: X' \to X$ be an \'{e}tale morphism from a normal surface singularity $(x' \in X')$ whose residue field $\kappa(x')$ is isomorphic to $K_1$ (Lemma \ref{residue extension}).
Then the dual graph of $(x' \in X')$ is a cusp with parameter $1$ and length $rn$:
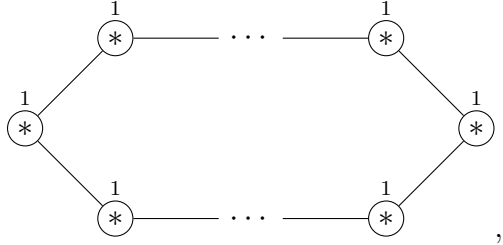
\[
\begin{tikzpicture}
\node[draw, shape=circle, inner sep=1.8pt] (A) at (0,0){$*$};
\draw[shift={(0,0.4)}] (0,0) node{\tiny $1$};
\node[draw, shape=circle, inner sep=1.8pt] (B) at (1.2,1.2){$*$};
\draw[shift={(0,0.4)}] (1.2,1.2) node{\tiny $1$};
\draw node (C) at (3,1.2){$\cdots$};
\node[draw, shape=circle, inner sep=1.8pt] (D) at (4.8 ,1.2){$*$};
\draw[shift={(0,0.4)}] (4.8,1.2) node{\tiny $1$};
\node[draw, shape=circle, inner sep=1.8pt] (E) at (6,0){$*$};
\draw[shift={(0,0.4)}] (6,0) node{\tiny $1$};
\node[draw, shape=circle, inner sep=1.8pt] (F) at (4.8,-1.2){$*$};
\draw[shift={(0,0.4)}] (4.8,-1.2) node{\tiny $1$};
\draw node (G) at (3,-1.2){$\cdots$};
\node[draw, shape=circle, inner sep=1.8pt] (H) at (1.2,-1.2){$*$};
\draw[shift={(0,0.4)}] (1.2,-1.2) node{\tiny $1$};

\draw (A) -- (B);
\draw (B) -- (C);
\draw (C) -- (D);
\draw (D) -- (E);
\draw (A) -- (H);
\draw (H) -- (G);
\draw (G) -- (F);
\draw (F) -- (E);
\end{tikzpicture},
\]
\end{enumerate}
\end{rem}

\begin{prop}\label{* for twisted cusp}
Let $(x \in X)$ be a normal surface singularity.
We assume that the following conditions are satisfied:
\begin{enumerate}[label=\textup{(\roman*)}]
\item $\ch(\kappa(x)) > 2$,
\item $(x \in X)$ is log canonical, and 
\item the dual graph of $(x \in X)$ is a twisted cusp with parameter $r \ge 1$ and length $n \ge 3$, that is, the dual graph is 
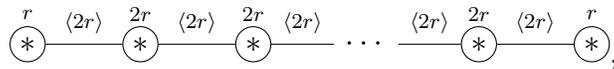
\[
\begin{tikzpicture}
\node[draw, shape=circle, inner sep=1.8pt] (X) at (-1.5,0){$*$};
\draw[shift={(0,0.4)}] (-1.5,0) node{\tiny $r$};
\node[draw, shape=circle, inner sep=1.8pt] (A) at (0,0){$*$};
\draw[shift={(0,0.4)}] (0,0) node{\tiny $2r$};
\node[draw, shape=circle, inner sep=1.8pt] (B) at (1.5,0){$*$};
\draw[shift={(0,0.4)}] (1.5,0) node{\tiny $2r$};
\draw node (C) at (3,0){$\cdots$};
\node[draw, shape=circle, inner sep=1.8pt] (D) at (4.5,0){$*$};
\draw[shift={(0,0.4)}] (4.5,0) node{\tiny $2r$};
\node[draw, shape=circle, inner sep=1.8pt] (E) at (6,0){$*$};
\draw[shift={(0,0.4)}] (6,0) node{\tiny $r$};

\draw (X) --node[auto=left]{\tiny $\langle 2r \rangle$} (A);
\draw (A) --node[auto=left]{\tiny $\langle 2r \rangle$} (B);
\draw (B) --node[auto=left]{\tiny $\langle 2r \rangle$} (C);
\draw (C) --node[auto=left]{\tiny $\langle 2r \rangle$} (D);
\draw (D) --node[auto=left]{\tiny $\langle 2r \rangle$} (E);
\end{tikzpicture},
\]
where the number of vertices is $n \ge 3$. (c.f. Notation \ref{notation}).
\end{enumerate}
Then there is an \'{e}tale morphism $\phi : X' \to X$ from a normal surface singularity $(x' \in X')$ whose dual graph is a cusp with parameter $1$:
\[
\begin{tikzpicture}
\node[draw, shape=circle, inner sep=1.8pt] (A) at (0,0){$*$};
\draw[shift={(0,0.4)}] (0,0) node{\tiny $1$};
\node[draw, shape=circle, inner sep=1.8pt] (B) at (1.2,1.2){$*$};
\draw[shift={(0,0.4)}] (1.2,1.2) node{\tiny $1$};
\draw node (C) at (3,1.2){$\cdots$};
\node[draw, shape=circle, inner sep=1.8pt] (D) at (4.8 ,1.2){$*$};
\draw[shift={(0,0.4)}] (4.8,1.2) node{\tiny $1$};
\node[draw, shape=circle, inner sep=1.8pt] (E) at (6,0){$*$};
\draw[shift={(0,0.4)}] (6,0) node{\tiny $1$};
\node[draw, shape=circle, inner sep=1.8pt] (F) at (4.8,-1.2){$*$};
\draw[shift={(0,0.4)}] (4.8,-1.2) node{\tiny $1$};
\draw node (G) at (3,-1.2){$\cdots$};
\node[draw, shape=circle, inner sep=1.8pt] (H) at (1.2,-1.2){$*$};
\draw[shift={(0,0.4)}] (1.2,-1.2) node{\tiny $1$};

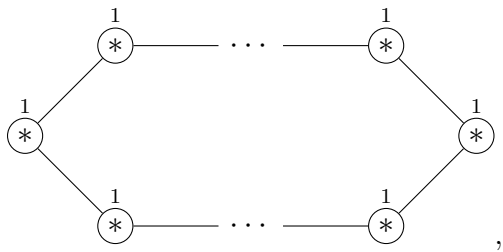
\draw (A) -- (B);
\draw (B) -- (C);
\draw (C) -- (D);
\draw (D) -- (E);
\draw (A) -- (H);
\draw (H) -- (G);
\draw (G) -- (F);
\draw (F) -- (E);
\end{tikzpicture},
\]

In particular, the minimal resolution $f: Y \to X$ satisfies the $(**)$-condition.
\end{prop}

\begin{proof}
As in the proof of Proposition \ref{* for rational}, we can take an \'{e}tale morphism $\phi: X' \to X$ from a normal surface singularity $(x' \in X')$ such that every exceptional prime divisor of the minimal resolution of $X'$ is geometrically irreducible over $\kappa(x')$.
It is enough to show that the dual graph $\Gamma$ of $(x' \in X')$ is a cusp with parameter $1$.

We first note that $(x' \in X')$ is again log canonical and not rational since $\phi$ is \'{e}tale.
Moreover, since the base change $f' : Y' \to X'$ of $f$ is the minimal resolution of $(x' \in X')$, the number of the vertexes of $\Gamma$ is at least three.
Therefore, it follows from the classification (Theorem \ref{classification} (3) below) that $\Gamma$ is either a cusp with length $\ge 3$ or a twisted cusp with length $\ge 3$.
In the former case, since every exceptional prime divisor in $Y'$ is geometrically irreducible, the parameter is one by Remark \ref{rmk on irr cusp}.

In the latter case, take an exceptional prime divisor $E_i \subseteq Y'$ such that the extension degree $[H^0(E_i, \sO_{E_i}) : \kappa(x')]$ is divisible by two.
It then follows from the assumption $\ch(\kappa(x)) \neq 2$ that $H^0(E_i, \sO_{E_i})$ is not purely inseparable over $\kappa(x')$.
Therefore, $E_i$ is not geometrically connected over $\kappa(x')$, which is a contradiction to the choice of $\phi$.
\end{proof}

\begin{prop}\label{* for cusp2}
Let $(x \in X)$ be a normal surface singularity.
We assume the following conditions are satisfied:
\begin{enumerate}[label=\textup{(\roman*)}]
\item $\ch(\kappa(x)) > 2$,
\item $(x \in X)$ is log canonical, and
\item the dual graph of $(x \in X)$ is a cusp with parameter $r \ge 1$ and length $2$, that is, the dual graph is
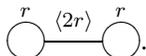
\[
\xygraph{
*=[o]++[Fo]{}([]!{+(0,+.25)} {{}^r})
-^{\langle 2r \rangle} [r] 
*=[o]++[Fo]{}([]!{+(0,+.25)} {{}^r})
}.\]
\end{enumerate}
Then the minimal resolution $f : Y \to X$ satisfies the $(**)$-condition.
\end{prop}

\begin{proof}
Let $f: Y \to X$ be a minimal resolution and $\Exc(f) = E_1 \cup E_2$ be the irreducible decomposition.
Since we have
\[
2r = (E_1 \cdot E_2) = \sum_{P \in E_1 \cap E_2} \dim_{\kappa(x)} \sO_{E_1 \cap E_2, P}
\]
and $\dim_{\kappa(x)} \sO_{E_1 \cap E_2 , P}$ is divisible by $r$, one of the following holds:
\begin{enumerate}[label=\textup{$(\arabic*)$}]
\item $E_1$ and $E_2$ transversally intersects at two points $P,Q$ such that 
\[
[\kappa(P):\kappa(x)] = [\kappa(Q): \kappa(P) ] = r.
\]
\item $E_1$ and $E_2$ transversally intersects at a single point $P$ such that 
\[
[\kappa(P):\kappa(x)] =2r.
\]
\item The scheme theoretic intersection $E_1 \cap E_2$ is non-reduced and $E_1 \cap E_2$ is a single point $P$ such that
\[
[\kappa(P): \kappa(x)] =r.
\]
\end{enumerate}

In the first case, the proof is similar to that of Proposition \ref{* for cusp}.
In the second case, the assertion follows from the similar argument as in Proposition \ref{* for twisted cusp}.

We consider the third case.
Since we have $K_{Y}+ (E_1+ E_2) = f^*K_X $ (c.f. Theorem \ref{classification} (3) below), the pair $(Y, E_1+E_2)$ is log canonical at $P$.
Therefore, it follows from \cite[Theorem 2.31]{Kol} (see also Remark \ref{rem: node}) that the pair $(Y,E_1+E_2)$ is SNC at $P$, which is a contradiction since $E_1 \cap E_2$ is non-reduced.
\end{proof}

\subsection{Simple elliptic singularities}

In this subsection, we give a sufficient condition for simple elliptic singularities to satisfy the $(**)$-condition.
We recall that a normal surface singularity is \emph{simple elliptic} if the dual graph is the following:
\[
\xygraph{
*=[o]++[Fo]{*}([]!{+(0,+.25)} {{}^1})([]!{+(0,-.3)} {{}_1}),
}
\]

\begin{defn}
Let $(x \in X)$ be a simple elliptic singularity and $E$ be the exceptional prime divisor of the minimal resolution.
We say that $(x \in X)$ is \emph{simple regular elliptic} if $E$ is regular, and is \emph{simple nodal elliptic} otherwise.
\end{defn}

\begin{prop}\label{* for simple elliptic}
Let $(x \in X)$ be a normal surface singularity.
We assume the following conditions are satisfied:
\begin{enumerate}[label=\textup{(\roman*)}]
\item $\ch(\kappa(x)) > 3$ and
\item $(x \in X)$ is log canonical and simple regular elliptic.
\end{enumerate}
Then the minimal resolution $f: Y \to X$ satisfies the $(*)$-condition.
\end{prop}

\begin{proof}
The assertion follows from Proposition \ref{smoothness of genus one}.
\end{proof}

\begin{prop}\label{* for simple nodal elliptic}
Let $(x \in X)$ be a normal surface singularity.
We assume the following conditions are satisfied:
\begin{enumerate}[label=\textup{(\roman*)}]
\item $\ch(\kappa(x)) > 3$,
\item $(x \in X)$ is log canonical and simple nodal elliptic.
\end{enumerate}
Then $X$ admits a resolution of singularities which satisfies the $(**)$-condition.
\end{prop}

\begin{proof}
Let $E$ be the exceptional divisor of the minimal resolution $f: Y \to X$.
We first consider the case where $E$ is not geometrically irreducible.
As in the proof of Proposition \ref{* for rational}, we may take an \'etale cover $X' \to X$, such that the exceptional locus of the minimal resolution $f \times_X \id_{X'} : Y \times_X X' \to X'$ of $X'$ is not irreducible.
Noting that $X'$ is again log canonical and non-rational, the dual graph of $X'$ is one of the graph in Figure \ref{classification lc non rational} below (see Theorem \ref{classification} (3) below).
Therefore, it follows from Propositions \ref{* for cusp}, \ref{* for twisted cusp}, \ref{* for cusp2} that the minimal resolution $f$ satisfies $(**)$-condition.

We next consider the case where $E$ is geometrically irreducible over $\kappa(x)$.
Since the pair $(Y,E)$ is log canonical, every singular point of $E$ is a node by \cite[Theorem 2.31]{Kol} (see also Remark \ref{rem: node}).
It then follows from Corollary \ref{separability genus one} that $E$ is geometrically reduced over $\kappa(x)$.
Therefore, by Proposition \ref{genus one}, there exists a non-regular $\kappa(x)$-rational point $P \in E$ such that $E$ is smooth over $\kappa(x)$ outside $P$.
Let $g: Z \to Y$ be the blowing up along $P$.
Then it is straightforward to show that the composite $f \circ g : Z \to X$ satisfies the $(*)$-condition.
\end{proof}

\begin{cor}\label{* for lc}
Let $(x \in X)$ be a normal surface singularity.
We assume the following conditions are satisfied:
\begin{enumerate}[label=\textup{(\roman*)}]
\item $\ch(\kappa(x)) > 3$ and
\item $(x \in X)$ is log canonical.
\end{enumerate}
Then there is a resolution $f: Y \to X$ which satisfies the $(**)$-condition.
\end{cor}

\begin{proof}
The assertion follows from the classification (Theorem \ref{classification} (2), (3) below) combining with Propositions \ref{* for rational}, \ref{* for cusp}, \ref{* for twisted cusp}, \ref{* for cusp2}, \ref{* for simple nodal elliptic} and \ref{* for simple elliptic}.
\end{proof}

\begin{cor}\label{geometric lc}
Let $X$ be a $2$-dimensional geometrically normal variety over a field $k$ of characteristic $p>3$ and $x_1, \dots, x_n \in X$ be the non-smooth points of $X$ over $k$.
Suppose that $X$ is log canonical at $x_i$ and the residue field $\kappa(x_i)$ of $X$ at $x_i$ is separable over $k$ for every $i$.
Then $X$ is geometrically log canonical over $k$.
\end{cor}

\begin{proof}
This follows from Lemma \ref{criterion2} and Corollary \ref{* for lc}.
\end{proof}

\section{Bertini type theorem for log canonical (resp.~klt) threefolds}\label{section bertini}

In this section, we prove the main theorem of this paper.

\begin{prop}\label{family}
Let $\pi: \mathcal{X} \to T$ be a quasi-projective morphism between varieties over an algebraically closed field $k$.
We assume the following conditions are satisfied:
\begin{enumerate}[label=\textup{(\roman*)}]
\item The geometric generic fiber $\mathcal{X}_{\overline{\eta}} : = \mathcal{X} \times_T \Spec(\overline{K(T)})$ admits a log resolution $g: Y \to \mathcal{X}_{\overline{\eta}}$ which is a projective morphism.
\item The generic fiber $\mathcal{X}_{\eta} : = \mathcal{X} \times_T \Spec(K(T))$ is geometrically log canonical (resp. geometrically klt) over $K(T)$.
\end{enumerate}
Then the fiber $\mathcal{X}_t$ is log canonical (resp.~klt) for a general closed point $t \in T(k)$.
\end{prop}

\begin{proof}
Since $\mathcal{X}_{\overline{\eta}}$ is quasi-projective over $\overline{K(T)}$ and $g$ is projective, there exist a finite field extension $L$ of $K(T)$ and a Cartesian diagram
\[
\xymatrix{
Y \ar^-{g}[r] \ar[d] & \mathcal{X}_{\overline{\eta}} \ar[d] \\
Y' \ar^-{h}[r] & X_{L} : = \mathcal{X}_{\eta} \times_{K(T)} L
,}\]
where $h$ is projective.
It follows from the faithfully flat descent of regularity (\cite[Theorem 23.7]{Mat}) that $h$ is a resolution of $X_L$.
After replacing $L$ by its finite extension, we may assume that every stratum $Z$ of $\Exc(h)$ with a reduced structure is geometrically integral over $L$.
We note that $Z$ is smooth over $L$ since the base change $Z \times_L \overline{K(T)}$ of a stratum $Z$ is isomorphic to some stratum of $\Exc(g)$.
After replacing $T$ by its normalization in $L$ and $\mathcal{X}$ by $\mathcal{X} \times_T S$, we may assume that there exists a projective log resolution $h: Y' \to \mathcal{X}_{\eta}$ such that the pair $(Y', \Exc(h))$ is geometrically SNC over $K(T)$.

On the other hand, it follows from the fact that $\mathcal{X}$ is quasi-projective over $T$ and $h$ is projective that there exists a non-empty open subset $U \subseteq T$ and a Cartesian diagram
\[
\xymatrix{
Y' \ar^-{h}[r] \ar[d] & \mathcal{X}_{\eta} \ar[d] \\
\mathcal{Y} \ar^-{\mu}[r] & f^{-1}(U),
}\]
where $\mu$ is a projective morphism.
Therefore, after shrinking $T$, we may assume that there exists a log resolution $\mu: \mathcal{Y} \to \mathcal{X}$ such that every stratum of the pair $(\mathcal{Y}, \Exc(\mu))$ is smooth over $T$ and maps surjectively onto $T$.
Moreover, since the generic fiber $\mathcal{X}_{\eta}$ is assumed to be $\Q$-Gorenstein and normal, we may shrink $T$ so that $T$ is smooth and $\mathcal{X}$ is normal and $\Q$-Gorenstein.

Since the locus where the fiber is geometrically normal is constructible (\cite[Proposition 9.9.4]{EGAIV3}), the fiber $\mathcal{X}_t$ is normal for a general closed point $t \in T$.
For such $t$, it follows from \cite[Exercise II. 5.16 (d) and Theorem 8.17]{Har} and Remark \ref{rem: dualizing} (1) that we have
\begin{eqnarray*}
\omega_{\mathcal{X}_t}  \cong  (\omega_{\mathcal{X}}|_{\mathcal{X}_t})^{**},
\end{eqnarray*}
where $^{**}$ denotes the reflexive hull.
Similarly, we have $\omega_{\mathcal{Y}_t}  \cong \omega_{\mathcal{Y}}|_{\mathcal{Y}_t}$, which implies that 
\[
(K_{\mathcal{Y}/\mathcal{X}})|_{\mathcal{Y}_t} = K_{\mathcal{Y}_t/\mathcal{X}_t}.
\]


Since the support of the relative canonical divisor $K_{\mathcal{Y}/\mathcal{X}}$ is contained in $\Exc(\mu)$, every irreducible component of the support is smooth over $T$ and maps surjectively onto $T$.
Therefore, the set of all coefficients of $K_{\mathcal{Y}_t/\mathcal{X}_t}$ coincides with that of $K_{\mathcal{Y}/\mathcal{X}}$.
The assertion now follows from Remark \ref{rmk on discrepancy}.
\end{proof}

\begin{thm}\label{bertini}
Let $X \subseteq \PP^N_k$ be a three dimensional normal quasi-projective variety over an algebraically closed field $k$ of characteristic $p>3$.
If $X$ has only log canonical (resp.~klt) singularities, then so does a general hyperplane section $H$ of $X$.
\end{thm}

\begin{proof}
Let $\PP^* : = \PP^N_k$ be the dual projective space and $\mathcal{Z} \subseteq \PP^N_k \times \PP^*$ be the universal family of hyperplanes of $\PP^N_k$, that is, the reduced closed subscheme of $\PP^N_k \times_k \PP^*$ such that the set of closed points of $\mathcal{Z}$ coincides with
\[
\{(x,H) \in \PP^N_k(k) \times (\PP^*)(k) \mid x \in H \}.
\]
Let $\mathcal{X}$ be the scheme theoretic intersection $\mathcal{Z} \cap (X \times \PP^*)$.
We note that a general hyperplane section of $X$ is none other than a general closed fiber of the second projection 
\[
p_2: \mathcal{X} \to \PP^*.
\]
Since the generic fiber $\mathcal{X}_{\eta}$ is a hypersurface of the three-dimensional variety $X \times_k \Spec(K(\PP^*))$ over $K(\PP^*)$, it is two-dimensional.
Given that any two-dimensional variety admits a log resolution which is a projective morphism, by Lemma \ref{family}, it suffices to show that the generic fiber $\mathcal{X}_{\eta}$ is geometrically normal and geometrically log canonical (resp. geometrically klt) over the function field $K(\PP^*)$.

Take a (not necessarily closed) point $P \in \mathcal{X}_{\eta}$ and we set $x: = \pi(P) \in X$, where $\pi: \mathcal{X}_{\eta} \to X$ is the morphism obtained by the first projection.
It then follows from the proof of \cite[Theorem 1]{CGM} that the fiber $\pi^{-1}(x)$ is geometrically regular both over $\kappa(x)$ and $K(\PP^*)$.

\[
\xymatrix{
\Spec(\kappa(x)) \ar[d] & & \pi^{-1}(x) \ar[ll] \ar[d] \\
X & \mathcal{X} \ar[l] \ar[d] & \mathcal{X}_{\eta} \ar[l] \ar[d] \\
& \PP^* & \Spec(K(\PP^*)) \ar[l]
}\]

By Corollary \ref{geometric klt} and Corollary \ref{geometric lc}, it suffices to prove that if $\mathcal{X}_{\eta}$ is not smooth over $K(\PP^*)$ at $P$, then the following properties hold:
\begin{enumerate}[label=\textup{(\alph*)}]
\item $P$ is a closed point of $\mathcal{X}_{\eta}$,
\item $\mathcal{X}_{\eta}$ is log canonical (resp.~klt) at $P$ and
\item the residue field $\kappa(P)$ is separable over $K(\PP^*)$.
\end{enumerate}

\textbf{Case. 1} : We first consider the case where $X$ is regular at $x$.
In this case, take a finite field extension $L$ of $K(\PP^*)$ and a point $Q \in \mathcal{X}_L : = \mathcal{X}_{\eta} \times_{K(\PP^*)} L$ over $P$.
Since the fiber of the first projection $\mathcal{X}_L \to X$ at $x$ is isomorphic to the regular scheme $\pi^{-1}(x) \times_{K(\PP^*)} L$, it follows from the ascent of regularity (\cite[Theorem 23.7]{Mat}) that $\mathcal{X}_L$ is regular at $Q$.
Therefore, $\mathcal{X}_{\eta}$ is smooth over $K(\PP^*)$ at $P$ and there is nothing to say.

\textbf{Case. 2} : We next consider the case where $X$ is non-regular at $x$.
Since $\pi: \mathcal{X}_{\eta} \to X$ is flat, we have the dimension formula (\cite[Lemma 00ON]{Sta}) 
\[
\dim(\sO_{\mathcal{X}_{\eta}, P}) = \dim(\sO_{X,x}) + \dim(\sO_{\pi^{-1}(x), P}),
\]
which implies that $\dim \sO_{X,x}=2$, $P$ is a closed point of $\mathcal{X}_{\eta}$ and $\pi^{-1}(x)$ consists of a single point $P$ around $P$.
Combining the third property with the fact that $\pi^{-1}(x)$ is geometrically regular over $K(\PP^*)$, the residue field $\kappa(P)$ is separable over $K(\PP^*)$.
Similarly, since $\pi^{-1}(x)$ is geometrically regular over $\kappa(x)$, it follows from \cite[Lemma 2.6 (3)]{ST mixed}\footnote{The assumption in loc.~cit. that the ring is in characteristic zero is unnecessary since the two dimensional scheme $\Spec (\sO_{X,x})$ admits a log resolution.} that $\mathcal{X}_{\eta}$ is log canonical (resp.~klt) at $P$, as desired. 
\end{proof}

\appendix

\section{Classification of dual graphs of numerically log canonical surface singularities}\label{appendix}

In this appendix, we list up the dual graphs of normal surface singularities which are log canonical (resp.~klt).
See \cite{Kol} for more details.

\begin{defn}\label{numerically sing}
Let $(x \in X)$ be a normal surface singularity, $f : Y \to X$ be the minimal resolution and $\Exc(f)=\bigcup_{i=1}^n E_i$ be the irreducible decomposition of the exceptional divisor.
Let $\Delta_Y$ be the unique $\Q$-divisor on $Y$ which satisfies
\[
(\Delta_Y \cdot E_i) = - (K_Y \cdot E_i)
\]
for all $i$.
We say that $(x \in X)$ is \emph{numerically log canonical} (resp.~\emph{numerically klt}) if every the coefficient of $\Delta_Y$ is at most (resp.~smaller than) $1$.
\end{defn}

\begin{defn}
A normal surface singularity $(x \in X)$ is \emph{rational} if the minimal resolution $f : Y \to X$ satisfies $R^1 f_* \sO_Y=0$.
\end{defn}

\begin{thm}[\cite{Kol}]\label{classification}
Let $(x \in X)$ be a normal surface singularity, $f: Y \to X$ be a minimal resolution, $\Gamma$ is the dual graph of $X$ (see Notation \ref{notation}), $E=\sum_{i=1}^n E_i$ be the sum of all exceptional prime divisors and $\Delta_Y$ be as in Definition \ref{numerically sing}.
Then the following holds:
\begin{enumerate}[label=\textup{(\arabic*)}]
\item $(x \in X)$ is numerically klt if and only if the dual graph $\Gamma$ coincides with one of the graphs in Figure \ref{classification klt} below.
\item $(x \in X)$ is rational and numerically log canonical but not numerically klt if and only if $\Gamma$ coincides with one of the graphs in Figure \ref{classification lc rational1} or Figure \ref{classification lc rational2} below.
\item The following are equivalent.
\begin{enumerate}[label=\textup{(\alph*)}]
\item $(x \in X)$ is numerically log canonical, but not a rational singularity.
\item $\Gamma$ coincides with one of the graph in Figure \ref{classification lc non rational} below.
\item $\Delta_Y=E$.
\end{enumerate}

\end{enumerate}
\end{thm}

\begin{proof}
We first prove the implication (c) $\Rightarrow$ (a) in (3).
Since we have $\omega_E \cong \omega_Y(E)|_E$, it follows from the assumption (c) that one has 
\[
\deg_{E/\kappa(x)}(\omega_E) : = \sum_{i=1}^n \deg_{E_i/\kappa(x)} (\omega_E|_{E_i}) = \sum_{i=1}^n ((K_Y+E) \cdot E_i) =0.
\]
Noting that Riemann-Roch theorem (Lemma \ref{RR} (1)) holds true even for reducible curve (\cite[Example 18.3.6]{Ful}), combining with Serre duality, we conclude that
\[
\deg_{E/\kappa(x)}(\omega_E) = 2 \chi(E/\kappa(x), \sO_E).
\]
Therefore, we have
\[
\dim_{\kappa(x)}(H^1(E, \sO_E))=\dim_{\kappa(x)}(H^0(E, \sO_E)) \neq 0.
\]
It then follows from the exact sequence
\[
R^1f_* \sO_Y \to H^1(E, \sO_E) \to R^2f_*\sO_Y(-E)=0
\]
that $(x \in X)$ is non-rational, which proves the implication (c) $\Rightarrow$ (a) in (3).
The converse implication (a) $\Rightarrow$ (c) follows from \cite[Theorem 2.28]{Kol}.
The rest of the proof follows from \cite[Subsection 3.3]{Kol}.
\end{proof}


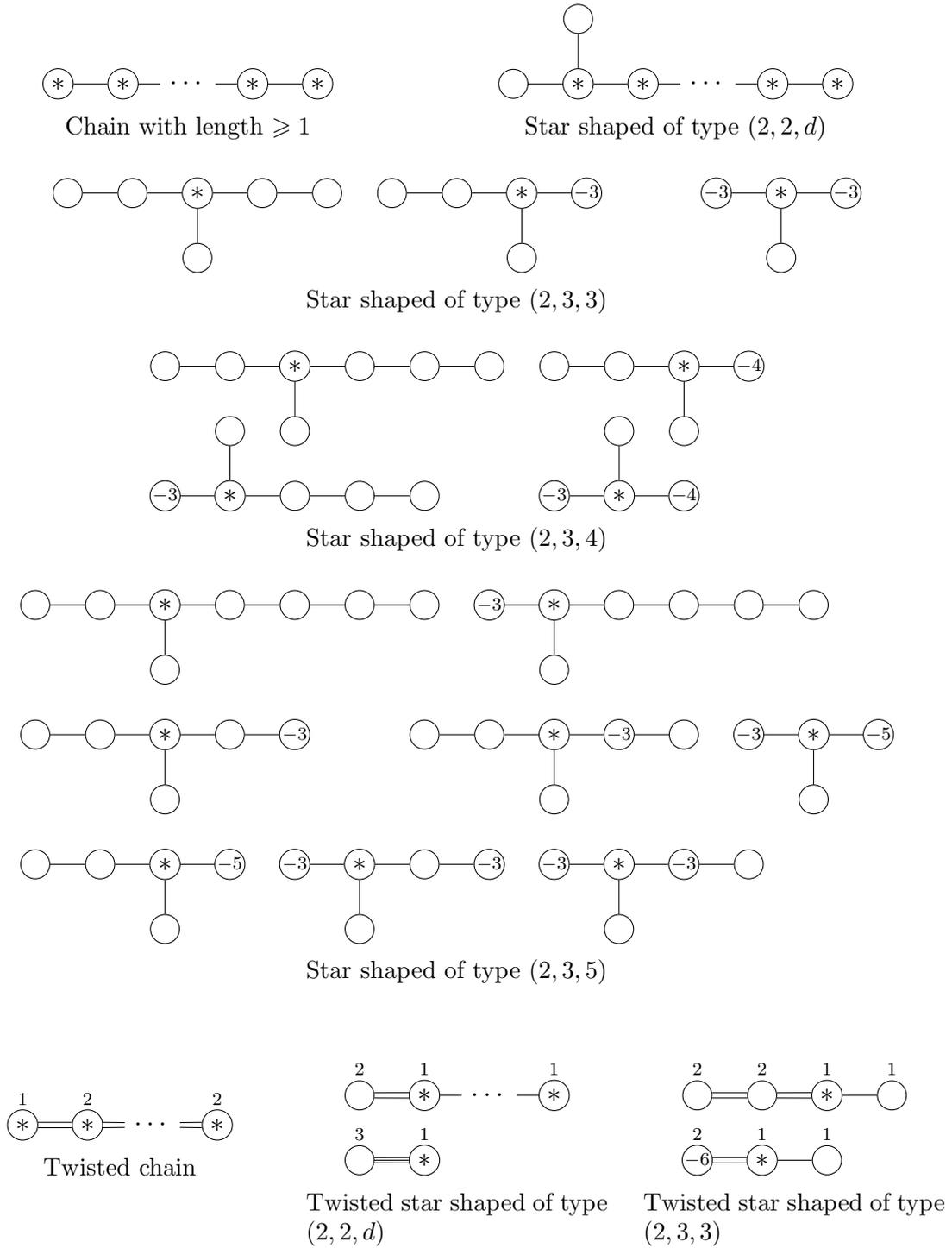
\begin{figure}[H]
\captionsetup[subfigure]{labelformat=empty}
\centering

\vspace*{0.5cm}
\begin{subfigure}[b]{0.4\textwidth}
\centering
\begin{tikzpicture}
\node[draw, shape=circle, inner sep=1.8pt] (C1) at (1,0){$*$};
\node[draw, shape=circle, inner sep=1.8pt] (C2) at (2,0){$*$};
\draw node (Dot) at (3,0){$\cdots$};
\node[draw, shape=circle, inner sep=1.8pt] (C3) at (4,0){$*$};
\node[draw, shape=circle, inner sep=1.8pt] (C4) at (5,0){$*$};

\draw (C1)--(C2);
\draw (C2)--(Dot);
\draw (Dot)--(C3);
\draw (C3)--(C4);
\end{tikzpicture}
\caption{Chain with length $\ge 1$}
\end{subfigure} 
\quad
\begin{subfigure}[b]{0.5\textwidth}
\centering
\begin{tikzpicture}
\node[draw, shape=circle, inner sep=4.5pt] (LU) at (0,0){};
\node[draw, shape=circle, inner sep=1.8pt] (C1) at (1,0){$*$};
\node[draw, shape=circle, inner sep=4.5pt] (LD) at (1,1){};
\node[draw, shape=circle, inner sep=1.8pt] (C2) at (2,0){$*$};
\draw node (Dot) at (3,0){$\cdots$};
\node[draw, shape=circle, inner sep=1.8pt] (C3) at (4,0){$*$};
\node[draw, shape=circle, inner sep=1.8pt] (C4) at (5,0){$*$};

\draw (LU)--(C1);
\draw (LD)--(C1);
\draw (C1)--(C2);
\draw (C2)--(Dot);
\draw (Dot)--(C3);
\draw (C3)--(C4);
\end{tikzpicture}
\caption{Star shaped of type $(2,2,d)$}
\end{subfigure} 

\vspace*{0.5cm}
\begin{subfigure}[b]{1\textwidth}
\centering
\begin{tikzpicture}
\node[draw, shape=circle, inner sep=4.5pt] (L2) at (0,0){};
\node[draw, shape=circle, inner sep=4.5pt] (L1) at (1,0){};
\node[draw, shape=circle, inner sep=1.8pt] (C) at (2,0){$*$};
\node[draw, shape=circle, inner sep=4.5pt] (D1) at (2,-1){};
\node[draw, shape=circle, inner sep=4.5pt] (R1) at (3,0){};
\node[draw, shape=circle, inner sep=4.5pt] (R2) at (4,0){};

\draw (L2)--(L1);
\draw (L1)--(C);
\draw (C)--(D1);
\draw (C)--(R1);
\draw (R1)--(R2);

\node[draw, shape=circle, inner sep=4.5pt] (2L2) at (5,0){};
\node[draw, shape=circle, inner sep=4.5pt] (2L1) at (6,0){};
\node[draw, shape=circle, inner sep=1.8pt] (2C) at (7,0){$*$};
\node[draw, shape=circle, inner sep=4.5pt] (2D1) at (7,-1){};
\node[draw, shape=circle, inner sep=0.3pt] (2R1) at (8,0){\tiny $-3$};

\draw (2L2)--(2L1);
\draw (2L1)--(2C);
\draw (2C)--(2D1);
\draw (2C)--(2R1);

\node[draw, shape=circle, inner sep=0.3pt] (3L1) at (10,0){\tiny $-3$};
\node[draw, shape=circle, inner sep=1.8pt] (3C) at (11,0){$*$};
\node[draw, shape=circle, inner sep=4.5pt] (3D1) at (11,-1){};
\node[draw, shape=circle, inner sep=0.3pt] (3R1) at (12,0){\tiny $-3$};

\draw (3L1)--(3C);
\draw (3C)--(3D1);
\draw (3C)--(3R1);
\end{tikzpicture}
\caption{Star shaped of type $(2,3,3)$}
\end{subfigure} 

\vspace*{0.5cm}
\begin{subfigure}[b]{1\textwidth}
\centering
\begin{tikzpicture}
\node[draw, shape=circle, inner sep=4.5pt] (L2) at (0,0){};
\node[draw, shape=circle, inner sep=4.5pt] (L1) at (1,0){};
\node[draw, shape=circle, inner sep=1.8pt] (C) at (2,0){$*$};
\node[draw, shape=circle, inner sep=4.5pt] (D1) at (2,-1){};
\node[draw, shape=circle, inner sep=4.5pt] (R1) at (3,0){};
\node[draw, shape=circle, inner sep=4.5pt] (R2) at (4,0){};
\node[draw, shape=circle, inner sep=4.5pt] (R3) at (5,0){};

\draw (L2)--(L1);
\draw (L1)--(C);
\draw (C)--(D1);
\draw (C)--(R1);
\draw (R1)--(R2);
\draw (R2)--(R3);

\node[draw, shape=circle, inner sep=4.5pt] (2L2) at (6,0){};
\node[draw, shape=circle, inner sep=4.5pt] (2L1) at (7,0){};
\node[draw, shape=circle, inner sep=1.8pt] (2C) at (8,0){$*$};
\node[draw, shape=circle, inner sep=4.5pt] (2D1) at (8,-1){};
\node[draw, shape=circle, inner sep=0.3pt] (2R1) at (9,0){\tiny $-4$};

\draw (2L2)--(2L1);
\draw (2L1)--(2C);
\draw (2C)--(2D1);
\draw (2C)--(2R1);

\node[draw, shape=circle, inner sep=0.3pt] (L1) at (0,-2){\tiny $-3$};
\node[draw, shape=circle, inner sep=1.8pt] (C) at (1,-2){$*$};
\node[draw, shape=circle, inner sep=4.5pt] (D1) at (1,-1){};
\node[draw, shape=circle, inner sep=4.5pt] (R1) at (2,-2){};
\node[draw, shape=circle, inner sep=4.5pt] (R2) at (3,-2){};
\node[draw, shape=circle, inner sep=4.5pt] (R3) at (4,-2){};

\draw (L1)--(C);
\draw (C)--(D1);
\draw (C)--(R1);
\draw (R1)--(R2);
\draw (R2)--(R3);

\node[draw, shape=circle, inner sep=0.3pt] (3L1) at (6,-2){\tiny $-3$};
\node[draw, shape=circle, inner sep=1.8pt] (3C) at (7,-2){$*$};
\node[draw, shape=circle, inner sep=4.5pt] (3D1) at (7,-1){};
\node[draw, shape=circle, inner sep=0.3pt] (3R1) at (8,-2){\tiny $-4$};

\draw (3L1)--(3C);
\draw (3C)--(3D1);
\draw (3C)--(3R1);
\end{tikzpicture}
\caption{Star shaped of type $(2,3,4)$}
\end{subfigure} 

\vspace*{0.5cm}
\begin{subfigure}[b]{1\textwidth}
\centering
\begin{tikzpicture}
\node[draw, shape=circle, inner sep=4.5pt] (L2) at (0,0){};
\node[draw, shape=circle, inner sep=4.5pt] (L1) at (1,0){};
\node[draw, shape=circle, inner sep=1.8pt] (C) at (2,0){$*$};
\node[draw, shape=circle, inner sep=4.5pt] (D1) at (2,-1){};
\node[draw, shape=circle, inner sep=4.5pt] (R1) at (3,0){};
\node[draw, shape=circle, inner sep=4.5pt] (R2) at (4,0){};
\node[draw, shape=circle, inner sep=4.5pt] (R3) at (5,0){};
\node[draw, shape=circle, inner sep=4.5pt] (R4) at (6,0){};

\draw (L2)--(L1);
\draw (L1)--(C);
\draw (C)--(D1);
\draw (C)--(R1);
\draw (R1)--(R2);
\draw (R2)--(R3);
\draw (R3)--(R4);

\node[draw, shape=circle, inner sep=0.3pt] (2L1) at (7,0){\tiny $-3$};
\node[draw, shape=circle, inner sep=1.8pt] (2C) at (8,0){$*$};
\node[draw, shape=circle, inner sep=4.5pt] (2D1) at (8,-1){};
\node[draw, shape=circle, inner sep=4.5pt] (2R1) at (9,0){};
\node[draw, shape=circle, inner sep=4.5pt] (2R2) at (10,0){};
\node[draw, shape=circle, inner sep=4.5pt] (2R3) at (11,0){};
\node[draw, shape=circle, inner sep=4.5pt] (2R4) at (12,0){};

\draw (2L1)--(2C);
\draw (2C)--(2D1);
\draw (2C)--(2R1);
\draw (2R1)--(2R2);
\draw (2R2)--(2R3);
\draw (2R3)--(2R4);

\node[draw, shape=circle, inner sep=4.5pt] (L2) at (0,-2){};
\node[draw, shape=circle, inner sep=4.5pt] (L1) at (1,-2){};
\node[draw, shape=circle, inner sep=1.8pt] (C) at (2,-2){$*$};
\node[draw, shape=circle, inner sep=4.5pt] (D1) at (2,-3){};
\node[draw, shape=circle, inner sep=4.5pt] (R1) at (3,-2){};
\node[draw, shape=circle, inner sep=0.3pt] (R2) at (4,-2){\tiny $-3$};

\draw (L2)--(L1);
\draw (L1)--(C);
\draw (C)--(D1);
\draw (C)--(R1);
\draw (R1)--(R2);

\node[draw, shape=circle, inner sep=4.5pt] (2L1) at (6,-2){};
\node[draw, shape=circle, inner sep=4.5pt] (2L2) at (7,-2){};
\node[draw, shape=circle, inner sep=1.8pt] (2C) at (8,-2){$*$};
\node[draw, shape=circle, inner sep=4.5pt] (2D1) at (8,-3){};
\node[draw, shape=circle, inner sep=0.3pt] (2R1) at (9,-2){\tiny $-3$};
\node[draw, shape=circle, inner sep=4.5pt] (2R2) at (10,-2){};

\draw (2L1)--(2L2);
\draw (2L2)--(2C);
\draw (2C)--(2D1);
\draw (2C)--(2R1);
\draw (2R1)--(2R2);

\node[draw, shape=circle, inner sep=0.3pt] (3L1) at (11,-2){\tiny $-3$};
\node[draw, shape=circle, inner sep=1.8pt] (3C) at (12,-2){$*$};
\node[draw, shape=circle, inner sep=4.5pt] (3D1) at (12,-3){};
\node[draw, shape=circle, inner sep=0.3pt] (3R1) at (13,-2){\tiny $-5$};

\draw (3L1)--(3C);
\draw (3C)--(3D1);
\draw (3C)--(3R1);

\node[draw, shape=circle, inner sep=4.5pt] (L2) at (0,-4){};
\node[draw, shape=circle, inner sep=4.5pt] (L1) at (1,-4){};
\node[draw, shape=circle, inner sep=1.8pt] (C) at (2,-4){$*$};
\node[draw, shape=circle, inner sep=4.5pt] (D1) at (2,-5){};
\node[draw, shape=circle, inner sep=0.3pt] (R1) at (3,-4){\tiny $-5$};

\draw (L2)--(L1);
\draw (L1)--(C);
\draw (C)--(D1);
\draw (C)--(R1);

\node[draw, shape=circle, inner sep=0.3pt] (2L1) at (4,-4){\tiny $-3$};
\node[draw, shape=circle, inner sep=1.8pt] (2C) at (5,-4){$*$};
\node[draw, shape=circle, inner sep=4.5pt] (2D1) at (5,-5){};
\node[draw, shape=circle, inner sep=4.5pt] (2R1) at (6,-4){};
\node[draw, shape=circle, inner sep=0.3pt] (2R2) at (7,-4){\tiny $-3$};

\draw (2L1)--(2C);
\draw (2C)--(2D1);
\draw (2C)--(2R1);
\draw (2R1)--(2R2);

\node[draw, shape=circle, inner sep=0.3pt] (3L1) at (8,-4){\tiny $-3$};
\node[draw, shape=circle, inner sep=1.8pt] (3C) at (9,-4){$*$};
\node[draw, shape=circle, inner sep=4.5pt] (3D1) at (9,-5){};
\node[draw, shape=circle, inner sep=0.3pt] (3R1) at (10,-4){\tiny $-3$};
\node[draw, shape=circle, inner sep=4.5pt] (3R2) at (11,-4){};

\draw (3L1)--(3C);
\draw (3C)--(3D1);
\draw (3C)--(3R1);
\draw (3R1)--(3R2);
\end{tikzpicture}
\caption{Star shaped of type $(2,3,5)$}
\end{subfigure} 

\vspace*{1cm}
\begin{subfigure}[b]{0.3\textwidth}
\centering
\begin{tikzpicture}
\node[draw, shape=circle, inner sep=1.8pt] (L1) at (0,0){$*$};
\draw[shift={(0,0.4)}] (0,0) node{\tiny $1$};
\node[draw, shape=circle, inner sep=1.8pt] (L2) at (1,0){$*$};
\draw[shift={(0,0.4)}] (1,0) node{\tiny $2$};
\draw node (Dot) at (2,0){$\cdots$};
\node[draw, shape=circle, inner sep=1.8pt] (L3) at (3,0){$*$};
\draw[shift={(0,0.4)}] (3,0) node{\tiny $2$};

\draw (L1.10)--(L2.170);
\draw (L1.-10)--(L2.190);
\draw (L2.10)--(Dot.175);
\draw (L2.-10)--(Dot.185);
\draw (Dot.5)--(L3.170);
\draw (Dot.-5)--(L3.190);
\end{tikzpicture}
\caption{Twisted chain}
\vspace*{1cm}
\end{subfigure} 
\quad
\begin{subfigure}[b]{0.3\textwidth}
\centering
\begin{tikzpicture}
\node[draw, shape=circle, inner sep=4.5pt] (A1) at (0,0){};
\draw[shift={(0,0.4)}] (0,0) node{\tiny $2$};
\node[draw, shape=circle, inner sep=1.8pt] (A2) at (1,0){$*$};
\draw[shift={(0,0.4)}] (1,0) node{\tiny $1$};
\draw node (ADot) at (2,0){$\cdots$};
\node[draw, shape=circle, inner sep=1.8pt] (A3) at (3,0){$*$};
\draw[shift={(0,0.4)}] (3,0) node{\tiny $1$};

\draw (A1.10)--(A2.170);
\draw (A1.-10)--(A2.190);
\draw (A2)--(ADot);
\draw (A3)--(ADot);

\node[draw, shape=circle, inner sep=4.5pt] (B1) at (0,-1){};
\draw[shift={(0,0.4)}] (0,-1) node{\tiny $3$};
\node[draw, shape=circle, inner sep=1.8pt] (B2) at (1,-1){$*$};
\draw[shift={(0,0.4)}] (1,-1) node{\tiny $1$};

\draw (B1.10)--(B2.170);
\draw (B1)--(B2);
\draw (B1.-10)--(B2.190);
\end{tikzpicture}
\caption{Twisted star shaped of type $(2,2,d)$}
\end{subfigure} 
\quad
\begin{subfigure}[b]{0.3\textwidth}
\centering
\begin{tikzpicture}
\node[draw, shape=circle, inner sep=4.5pt] (A1) at (0,0){};
\draw[shift={(0,0.4)}] (0,0) node{\tiny $2$};
\node[draw, shape=circle, inner sep=4.5pt] (A2) at (1,0){};
\draw[shift={(0,0.4)}] (1,0) node{\tiny $2$};
\node[draw, shape=circle, inner sep=1.8pt] (A3) at (2,0){$*$};
\draw[shift={(0,0.4)}] (2,0) node{\tiny $1$};
\node[draw, shape=circle, inner sep=4.5pt] (A4) at (3,0){};
\draw[shift={(0,0.4)}] (3,0) node{\tiny $1$};

\draw (A1.10)--(A2.170);
\draw (A1.-10)--(A2.190);
\draw (A2.10)--(A3.170);
\draw (A2.-10)--(A3.190);
\draw (A3)--(A4);

\node[draw, shape=circle, inner sep=0.3pt] (B1) at (0,-1){\tiny $-6$};
\draw[shift={(0,0.4)}] (0, -1) node{\tiny $2$};
\node[draw, shape=circle, inner sep=1.8pt] (B2) at (1,-1){$*$};
\draw[shift={(0,0.4)}] (1,-1) node{\tiny $1$};
\node[draw, shape=circle, inner sep=4.5pt] (B3) at (2,-1){};
\draw[shift={(0,0.4)}] (2,-1) node{\tiny $1$};

\draw (B1.10)--(B2.170);
\draw (B1.-10)--(B2.190);
\draw (B2)--(B3);
\end{tikzpicture}
\caption{Twisted star shaped of type $(2,3,3)$}
\end{subfigure} 
\caption{List of dual graphs of numerical klt singularities}
\label{classification klt}
\end{figure}



\begin{figure}[H]
\captionsetup[subfigure]{labelformat=empty}
\centering

\vspace{1cm}
\begin{subfigure}[b]{1\textwidth}
\centering
\begin{tikzpicture}
\node[draw, shape=circle, inner sep=4.5pt] (L2) at (0,0){};
\node[draw, shape=circle, inner sep=4.5pt] (L1) at (1,0){};
\node[draw, shape=circle, inner sep=1.8pt] (C) at (2,0){$*$};
\node[draw, shape=circle, inner sep=4.5pt] (D1) at (2,-1){};
\node[draw, shape=circle, inner sep=4.5pt] (R1) at (3,0){};
\node[draw, shape=circle, inner sep=4.5pt] (R2) at (4,0){};
\node[draw, shape=circle, inner sep=4.5pt] (R3) at (5,0){};
\node[draw, shape=circle, inner sep=4.5pt] (R4) at (6,0){};
\node[draw, shape=circle, inner sep=4.5pt] (R5) at (7,0){};

\draw (L2)--(L1);
\draw (L1)--(C);
\draw (C)--(D1);
\draw (C)--(R1);
\draw (R1)--(R2);
\draw (R2)--(R3);
\draw (R3)--(R4);
\draw (R4)--(R5);

\node[draw, shape=circle, inner sep=4.5pt] (LL2) at (7,-1){};
\node[draw, shape=circle, inner sep=4.5pt] (LL1) at (8,-1){};
\node[draw, shape=circle, inner sep=1.8pt] (CC) at (9,-1){$*$};
\node[draw, shape=circle, inner sep=4.5pt] (DD1) at (9,0){};
\node[draw, shape=circle, inner sep=0.3pt] (RR1) at (10,-1){\tiny $-6$};

\draw (LL2)--(LL1);
\draw (LL1)--(CC);
\draw (CC)--(DD1);
\draw (CC)--(RR1);

\node[draw, shape=circle, inner sep=0.3pt] (LLL1) at (0,-3){\tiny $-3$};
\node[draw, shape=circle, inner sep=1.8pt] (CCC) at (1,-3){$*$};
\node[draw, shape=circle, inner sep=4.5pt] (DDD1) at (1,-2){};
\node[draw, shape=circle, inner sep=4.5pt] (RRR1) at (2,-3){};
\node[draw, shape=circle, inner sep=4.5pt] (RRR2) at (3,-3){};
\node[draw, shape=circle, inner sep=4.5pt] (RRR3) at (4,-3){};
\node[draw, shape=circle, inner sep=4.5pt] (RRR4) at (5,-3){};
\node[draw, shape=circle, inner sep=4.5pt] (RRR5) at (6,-3){};

\draw (LLL1)--(CCC);
\draw (CCC)--(DDD1);
\draw (CCC)--(RRR1);
\draw (RRR1)--(RRR2);
\draw (RRR2)--(RRR3);
\draw (RRR3)--(RRR4);
\draw (RRR4)--(RRR5);

\node[draw, shape=circle, inner sep=0.3pt] (LLLL1) at (8,-3){\tiny $-3$};
\node[draw, shape=circle, inner sep=1.8pt] (CCCC) at (9,-3){$*$};
\node[draw, shape=circle, inner sep=4.5pt] (DDDD1) at (9,-2){};
\node[draw, shape=circle, inner sep=0.3pt] (RRRR1) at (10,-3){\tiny $-6$};

\draw (LLLL1)--(CCCC);
\draw (RRRR1)--(CCCC);
\draw (DDDD1)--(CCCC);

\end{tikzpicture}
\caption{Star shaped of type $(2,3,6)$}
\end{subfigure}

\vspace{1cm}
\begin{subfigure}[b]{1\textwidth}
\centering
\begin{tikzpicture}
\node[draw, shape=circle, inner sep=4.5pt] (L2) at (0,0){};
\node[draw, shape=circle, inner sep=4.5pt] (L1) at (1,0){};
\node[draw, shape=circle, inner sep=1.8pt] (C) at (2,0){$*$};
\node[draw, shape=circle, inner sep=0.3pt] (D1) at (2,-1){\tiny $-3$};
\node[draw, shape=circle, inner sep=4.5pt] (R1) at (3,0){};
\node[draw, shape=circle, inner sep=4.5pt] (R2) at (4,0){};

\draw (L2)--(L1);
\draw (L1)--(C);
\draw (C)--(D1);
\draw (C)--(R1);
\draw (R1)--(R2);

\node[draw, shape=circle, inner sep=4.5pt] (LL2) at (6,0){};
\node[draw, shape=circle, inner sep=4.5pt] (LL1) at (7,0){};
\node[draw, shape=circle, inner sep=1.8pt] (CC) at (8,0){$*$};
\node[draw, shape=circle, inner sep=4.5pt] (DD1) at (8,-1){};
\node[draw, shape=circle, inner sep=4.5pt] (DD2) at (8,-2){};
\node[draw, shape=circle, inner sep=4.5pt] (RR1) at (9,0){};
\node[draw, shape=circle, inner sep=4.5pt] (RR2) at (10,0){};

\draw (LL2)--(LL1);
\draw (LL1)--(CC);
\draw (CC)--(DD1);
\draw (DD1)--(DD2);
\draw (CC)--(RR1);
\draw (RR1)--(RR2);

\node[draw, shape=circle, inner sep=4.5pt] (LLL1) at (0,-3){};
\node[draw, shape=circle, inner sep=4.5pt] (LLL2) at (1,-3){};
\node[draw, shape=circle, inner sep=1.8pt] (CCC) at (2,-3){$*$};
\node[draw, shape=circle, inner sep=0.3pt] (DDD1) at (2,-2){\tiny $-3$};
\node[draw, shape=circle, inner sep=0.3pt] (RRR1) at (3,-3){\tiny $-3$};

\draw (LLL1)--(LLL2);
\draw (LLL2)--(CCC);
\draw (CCC)--(DDD1);
\draw (CCC)--(RRR1);

\node[draw, shape=circle, inner sep=0.3pt] (LLLL1) at (5,-3){\tiny $-3$};
\node[draw, shape=circle, inner sep=1.8pt] (CCCC) at (6,-3){$*$};
\node[draw, shape=circle, inner sep=0.3pt] (DDDD1) at (6,-2){\tiny $-3$};
\node[draw, shape=circle, inner sep=0.3pt] (RRRR1) at (7,-3){\tiny $-3$};

\draw (LLLL1)--(CCCC);
\draw (RRRR1)--(CCCC);
\draw (DDDD1)--(CCCC);

\end{tikzpicture}
\caption{Star shaped of type $(3,3,3)$}
\end{subfigure} 

\vspace{1cm}
\begin{subfigure}[b]{1\textwidth}
\centering
\begin{tikzpicture}
\node[draw, shape=circle, inner sep=4.5pt] (L1) at (0,0){};
\node[draw, shape=circle, inner sep=4.5pt] (L2) at (1,0){};
\node[draw, shape=circle, inner sep=4.5pt] (L3) at (2,0){};
\node[draw, shape=circle, inner sep=1.8pt] (C) at (3,0){$*$};
\node[draw, shape=circle, inner sep=4.5pt] (D1) at (3,-1){};
\node[draw, shape=circle, inner sep=4.5pt] (R1) at (4,0){};
\node[draw, shape=circle, inner sep=4.5pt] (R2) at (5,0){};
\node[draw, shape=circle, inner sep=4.5pt] (R3) at (6,0){};

\draw (L1)--(L2);
\draw (L2)--(L3);
\draw (L3)--(C);
\draw (C)--(D1);
\draw (C)--(R1);
\draw (R1)--(R2);
\draw (R2)--(R3);

\node[draw, shape=circle, inner sep=4.5pt] (LLL1) at (6,-1){};
\node[draw, shape=circle, inner sep=4.5pt] (LLL2) at (7,-1){};
\node[draw, shape=circle, inner sep=4.5pt] (LLL3) at (8,-1){};
\node[draw, shape=circle, inner sep=1.8pt] (CCC) at (9,-1){$*$};
\node[draw, shape=circle, inner sep=0.3pt] (DDD1) at (10,-1){\tiny $-4$};
\node[draw, shape=circle, inner sep=4.5pt] (UUU1) at (9,0){};

\draw (LLL1)--(LLL2);
\draw (LLL2)--(LLL3);
\draw (LLL3)--(CCC);
\draw (CCC)--(DDD1);
\draw (CCC)--(UUU1);

\node[draw, shape=circle, inner sep=0.3pt] (LL1) at (4,-2){\tiny $-4$};
\node[draw, shape=circle, inner sep=1.8pt] (CC) at (5,-2){$*$};
\node[draw, shape=circle, inner sep=4.5pt] (DD1) at (5,-1){};
\node[draw, shape=circle, inner sep=0.3pt] (RR1) at (6,-2){\tiny $-4$};

\draw (LL1)--(CC);
\draw (CC)--(DD1);
\draw (CC)--(RR1);

\end{tikzpicture}
\caption{Star shaped of type $(2,4,4)$}
\end{subfigure} 

\vspace{1cm}


\begin{subfigure}[b]{1\textwidth}
\centering
\begin{tikzpicture}
\node[draw, shape=circle, inner sep=4.5pt] (LU) at (0,0){};
\node[draw, shape=circle, inner sep=1.8pt] (C1) at (1,0){$*$};
\node[draw, shape=circle, inner sep=4.5pt] (LD) at (1,-1){};
\node[draw, shape=circle, inner sep=1.8pt] (C2) at (2,0){$*$};
\draw node (Dot) at (3,0){$\cdots$};
\node[draw, shape=circle, inner sep=1.8pt] (C3) at (4,0){$*$};
\node[draw, shape=circle, inner sep=1.8pt] (C4) at (5,0){$*$};
\node[draw, shape=circle, inner sep=4.5pt] (RU) at (6,0){};
\node[draw, shape=circle, inner sep=4.5pt] (RD) at (5,-1){};

\draw (LU)--(C1);
\draw (LD)--(C1);
\draw (C1)--(C2);
\draw (C2)--(Dot);
\draw (Dot)--(C3);
\draw (C3)--(C4);
\draw (RU)--(C4);
\draw (RD)--(C4);
\end{tikzpicture}
\caption{$\tilde{D}_{n}$ with $n \ge 4$}
\end{subfigure} 

\vspace{1cm}

\caption{List of non-twisted dual graphs of rational numerical log canonical singularities}
\label{classification lc rational1}
\end{figure}

\begin{figure}[H]
\captionsetup[subfigure]{labelformat=empty}
\centering

\begin{subfigure}[b]{1\textwidth}
\centering
\begin{tikzpicture}
\node[draw, shape=circle, inner sep=4.5pt] (A1) at (0,0){};
\draw[shift={(0,0.4)}] (0,0) node{\tiny $2$};
\node[draw, shape=circle, inner sep=4.5pt] (A2) at (1,0){};
\draw[shift={(0,0.4)}] (1,0) node{\tiny $2$};
\node[draw, shape=circle, inner sep=1.8pt] (A3) at (2,0){$*$};
\draw[shift={(0,0.4)}] (2,0) node{\tiny $1$};
\node[draw, shape=circle, inner sep=0.3pt] (A4) at (3,0){\tiny $-3$};
\draw[shift={(0,0.4)}] (3,0) node{\tiny $1$};

\draw (A1.10)--(A2.170);
\draw (A1.-10)--(A2.190);
\draw (A2.10)--(A3.170);
\draw (A2.-10)--(A3.190);
\draw (A3)--(A4);

\node[draw, shape=circle, inner sep=4.5pt] (B1) at (4,0){};
\draw[shift={(0,0.4)}] (4,0) node{\tiny $2$};
\node[draw, shape=circle, inner sep=4.5pt] (B2) at (5,0){};
\draw[shift={(0,0.4)}] (5,0) node{\tiny $2$};
\node[draw, shape=circle, inner sep=1.8pt] (B3) at (6,0){$*$};
\draw[shift={(0,0.4)}] (6,0) node{\tiny $1$};
\node[draw, shape=circle, inner sep=4.5pt] (B4) at (7,0){};
\draw[shift={(0,0.4)}] (7,0) node{\tiny $1$};
\node[draw, shape=circle, inner sep=4.5pt] (B5) at (8,0){};
\draw[shift={(0,0.4)}] (8,0) node{\tiny $1$};

\draw (B1.10)--(B2.170);
\draw (B1.-10)--(B2.190);
\draw (B2.10)--(B3.170);
\draw (B2.-10)--(B3.190);
\draw (B3)--(B4);
\draw (B4)--(B5);

\node[draw, shape=circle, inner sep=4.5pt] (C1) at (9,0){};
\draw[shift={(0,0.4)}] (9,0) node{\tiny $3$};
\node[draw, shape=circle, inner sep=4.5pt] (C2) at (10,0){};
\draw[shift={(0,0.4)}] (10,0) node{\tiny $3$};
\node[draw, shape=circle, inner sep=1.8pt] (C3) at (11,0){$*$};
\draw[shift={(0,0.4)}] (11,0) node{\tiny $1$};

\draw (C1)--(C2);
\draw (C1.10)--(C2.170);
\draw (C1.-10)--(C2.190);
\draw (C2)--(C3);
\draw (C2.10)--(C3.170);
\draw (C2.-10)--(C3.190);

\node[draw, shape=circle, inner sep=4.5pt] (A1) at (0,-1){};
\draw[shift={(0,0.4)}] (0,-1) node{\tiny $1$};
\node[draw, shape=circle, inner sep=4.5pt] (A2) at (1,-1){};
\draw[shift={(0,0.4)}] (1,-1) node{\tiny $1$};
\node[draw, shape=circle, inner sep=1.8pt] (A3) at (2,-1){$*$};
\draw[shift={(0,0.4)}] (2,-1) node{\tiny $1$};
\node[draw, shape=circle, inner sep=0.3pt] (A4) at (3,-1){\tiny $-6$};
\draw[shift={(0,0.4)}] (3,-1) node{\tiny $2$};

\draw (A1)--(A2);
\draw (A2)--(A3);
\draw (A3.10)--(A4.170);
\draw (A3.-10)--(A4.190);

\node[draw, shape=circle, inner sep=0.3pt] (B1) at (5,-1){\tiny $-6$};
\draw[shift={(0,0.4)}] (5,-1) node{\tiny $2$};
\node[draw, shape=circle, inner sep=1.8pt] (B2) at (6,-1){$*$};
\draw[shift={(0,0.4)}] (6,-1) node{\tiny $1$};
\node[draw, shape=circle, inner sep=0.3pt] (B3) at (7,-1){\tiny $-3$};
\draw[shift={(0,0.4)}] (7,-1) node{\tiny $1$};

\draw (B1.10)--(B2.170);
\draw (B1.-10)--(B2.190);
\draw (B2)--(B3);

\node[draw, shape=circle, inner sep=0.3pt] (C1) at (9,-1){\tiny $-9$};
\draw[shift={(0,0.4)}] (9,-1) node{\tiny $3$};
\node[draw, shape=circle, inner sep=1.8pt] (C2) at (10,-1){$*$};
\draw[shift={(0,0.4)}] (10,-1) node{\tiny $1$};

\draw (C1)--(C2);
\draw (C1.10)--(C2.170);
\draw (C1.-10)--(C2.190);

\end{tikzpicture}
\caption{Twisted star shaped of type $(3,3,3)$}
\end{subfigure}
 
\begin{subfigure}[b]{1\textwidth}
\centering
\begin{tikzpicture}
\node[draw, shape=circle, inner sep=4.5pt] (A1) at (0,0){};
\draw[shift={(0,0.4)}] (0,0) node{\tiny $2$};
\node[draw, shape=circle, inner sep=4.5pt] (A2) at (1,0){};
\draw[shift={(0,0.4)}] (1,0) node{\tiny $2$};
\node[draw, shape=circle, inner sep=4.5pt] (A3) at (2,0){};
\draw[shift={(0,0.4)}] (2,0) node{\tiny $2$};
\node[draw, shape=circle, inner sep=1.8pt] (A4) at (3,0){$*$};
\draw[shift={(0,0.4)}] (3,0) node{\tiny $1$};
\node[draw, shape=circle, inner sep=4.5pt] (A5) at (4,0){};
\draw[shift={(0,0.4)}] (4,0) node{\tiny $1$};

\draw (A1.10)--(A2.170);
\draw (A1.-10)--(A2.190);
\draw (A2.10)--(A3.170);
\draw (A2.-10)--(A3.190);
\draw (A3.10)--(A4.170);
\draw (A3.-10)--(A4.190);
\draw (A4)--(A5);

\node[draw, shape=circle, inner sep=4.5pt] (B1) at (0, -1){};
\draw[shift={(0,0.4)}] (0,-1) node{\tiny $1$};
\node[draw, shape=circle, inner sep=1.8pt] (B2) at (1,-1){$*$};
\draw[shift={(0,0.4)}] (1,-1) node{\tiny $1$};
\node[draw, shape=circle, inner sep=0.3pt] (B3) at (2,-1){\tiny $-8$};
\draw[shift={(0,0.4)}] (2,-1) node{\tiny $2$};

\draw (B1)--(B2);
\draw (B2.10)--(B3.170);
\draw (B2.-10)--(B3.190);
\end{tikzpicture}
\caption{Twisted star shaped of type $(2,4,4)$}
\end{subfigure} 


\begin{subfigure}[b]{1\textwidth}
\centering
\begin{tikzpicture}
\node[draw, shape=circle, inner sep=1.8pt] (L1) at (0,0){$*$};
\draw[shift={(0,0.4)}] (0,0) node{\tiny $1$};
\node[draw, shape=circle, inner sep=1.8pt] (L2) at (1,0){$*$};
\draw[shift={(0,0.4)}] (1,0) node{\tiny $2$};
\draw node (Dot) at (2,0){$\cdots$};
\node[draw, shape=circle, inner sep=1.8pt] (L3) at (3,0){$*$};
\draw[shift={(0,0.4)}] (3,0) node{\tiny $2$};
\node[draw, shape=circle, inner sep=1.8pt] (L4) at (4,0){$*$};
\draw[shift={(0.2,0.4)}] (4,0) node{\tiny $2$};
\node[draw, shape=circle, inner sep=4.5pt] (U1) at (4,1){};
\draw[shift={(0,0.4)}] (4,1) node{\tiny $2$};
\node[draw, shape=circle, inner sep=4.5pt] (D1) at (5,0){};
\draw[shift={(0,0.4)}] (5,0) node{\tiny $2$};

\draw (L1.10)--(L2.170);
\draw (L1.-10)--(L2.190);
\draw (L2.10)--(Dot.175);
\draw (L2.-10)--(Dot.185);
\draw (Dot.5)--(L3.170);
\draw (Dot.-5)--(L3.190);
\draw (L3.10)--(L4.170);
\draw (L3.-10)--(L4.190);
\draw (L4.100)--(U1.-100);
\draw (L4.80)--(U1.-80);
\draw (L4.10)--(D1.170);
\draw (L4.-10)--(D1.190);

\node[draw, shape=circle, inner sep=4.5pt] (LL1) at (6,0){};
\draw[shift={(0,0.4)}] (6, 0) node{\tiny $2$};
\node[draw, shape=circle, inner sep=1.8pt] (LL2) at (7,0){$*$};
\draw[shift={(0,0.4)}] (7,0) node{\tiny $1$};
\draw node (Dott) at (8,0){$\cdots$};
\node[draw, shape=circle, inner sep=1.8pt] (LL3) at (9,0){$*$};
\draw[shift={(0,0.4)}] (9,0) node{\tiny $1$};
\node[draw, shape=circle, inner sep=1.8pt] (LL4) at (10,0){$*$};
\draw[shift={(0.2,0.4)}] (10,0) node{\tiny $1$};
\node[draw, shape=circle, inner sep=4.5pt] (UU1) at (10,1){};
\draw[shift={(0,0.4)}] (10,1) node{\tiny $1$};
\node[draw, shape=circle, inner sep=4.5pt] (DD1) at (11,0){};
\draw[shift={(0, 0.4)}] (11,0) node{\tiny $1$};

\draw (LL1.10)--(LL2.170);
\draw (LL1.-10)--(LL2.190);
\draw (LL2) --(Dott);
\draw (Dott) --(LL3);
\draw (LL3) --(LL4);
\draw (LL4) --(UU1);
\draw (LL4) --(DD1);

\node[draw, shape=circle, inner sep=4.5pt] (A1) at (0,-1){};
\draw[shift={(0,0.4)}] (0,-1) node{\tiny $4$};
\node[draw, shape=circle, inner sep=1.8pt] (A2) at (1,-1){$*$};
\draw[shift={(0,0.4)}] (1,-1) node{\tiny $2$};
\draw node (ADot) at (2,-1){$\cdots$};
\node[draw, shape=circle, inner sep=1.8pt] (A3) at (3,-1){$*$};
\draw[shift={(0,0.4)}] (3,-1) node{\tiny $2$};
\node[draw, shape=circle, inner sep=1.8pt] (A4) at (4,-1){$*$};
\draw[shift={(0,0.4)}] (4,-1) node{\tiny $1$};

\draw (A1) --node[auto=left]{\tiny $\langle 4 \rangle$} (A2);
\draw (A2.10)--(ADot.175);
\draw (A2.-10)--(ADot.185);
\draw (ADot.5)--(A3.170);
\draw (ADot.-5)--(A3.190);
\draw (A3.10)--(A4.170);
\draw (A3.-10)--(A4.190);

\node[draw, shape=circle, inner sep=4.5pt] (B1) at (6,-1){};
\draw[shift={(0,0.4)}] (6, -1) node{\tiny $2$};
\node[draw, shape=circle, inner sep=1.8pt] (B2) at (7,-1){$*$};
\draw[shift={(0,0.4)}] (7,-1) node{\tiny $1$};
\draw node (BDot) at (8,-1){$\cdots$};
\node[draw, shape=circle, inner sep=1.8pt] (B3) at (9,-1){$*$};
\draw[shift={(0,0.4)}] (9,-1) node{\tiny $1$};
\node[draw, shape=circle, inner sep=4.5pt] (B4) at (10,-1){};
\draw[shift={(0,0.4)}] (10,-1) node{\tiny $2$};

\draw (B1.10)--(B2.170);
\draw (B1.-10)--(B2.190);
\draw (B2)--(BDot);
\draw (B3)--(BDot);
\draw (B3.10)--(B4.170);
\draw (B3.-10)--(B4.190);

\node[draw, shape=circle, inner sep=1.8pt] (A1) at (0,-2){$*$};
\draw[shift={(0,0.4)}] (0,-2) node{\tiny $1$};
\node[draw, shape=circle, inner sep=4.5pt] (A2) at (1,-2){};
\draw[shift={(0,0.4)}] (1,-2) node{\tiny $4$};

\draw (A1) --node[auto=left]{\tiny $\langle 4 \rangle$} (A2);

\node[draw, shape=circle, inner sep=4.5pt] (B1) at (6,-2){};
\draw[shift={(0,0.4)}] (6,-2) node{\tiny $1$};
\node[draw, shape=circle, inner sep=1.8pt] (B2) at (7,-2){$*$};
\draw[shift={(0,0.4)}] (7,-2) node{\tiny $1$};
\node[draw, shape=circle, inner sep=4.5pt] (B3) at (8,-2){};
\draw[shift={(0,0.4)}] (8,-2) node{\tiny $3$};

\draw (B1)--(B2);
\draw (B2)--(B3);
\draw (B2.10)--(B3.170);
\draw (B2.-10)--(B3.190);

\end{tikzpicture}

\caption{Twisted $\tilde{D}_{n}$}
\end{subfigure}
\caption{List of twisted dual graphs of rational numerical log canonical singularities}
\label{classification lc rational2}
\end{figure}
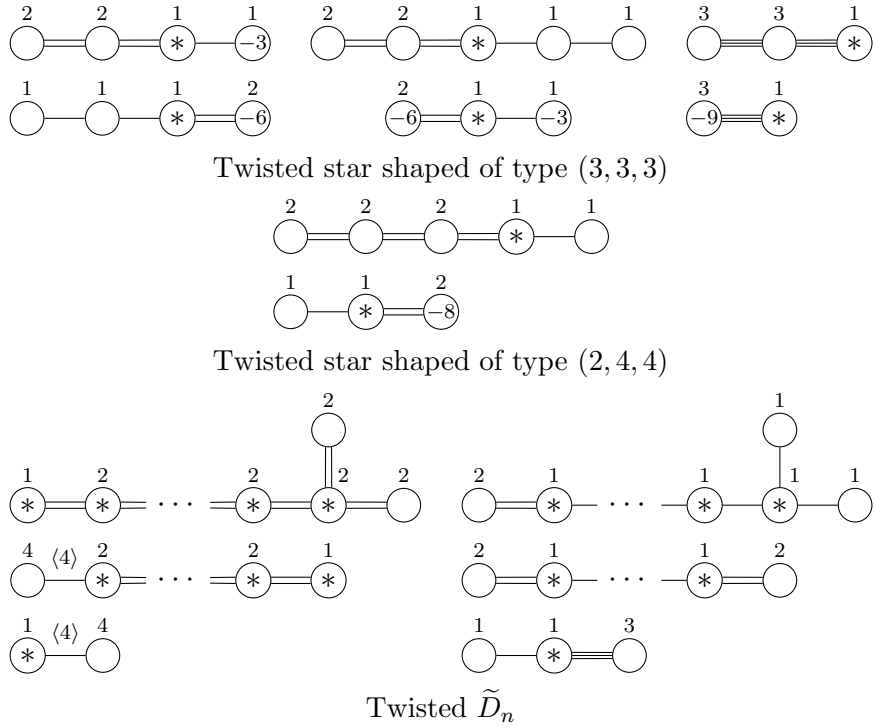


\begin{figure}[H]
\captionsetup[subfigure]{labelformat=empty}
\centering

\begin{subfigure}[b]{0.3\textwidth}
\centering
\begin{tikzpicture}
\node[draw, shape=circle, inner sep=1.8pt] at (0,0){{$*$}};
\draw[shift={(0,0.4)}] (0,0) node{\tiny $1$};
\draw[shift={(0,-0.4)}] (0,0) node{\tiny $1$};
\end{tikzpicture}
\caption{Simple elliptic}
\label{Simple elliptic}
\end{subfigure} 
\quad
\begin{subfigure}[b]{0.6\textwidth}
\centering
\begin{tikzpicture}
\node[draw, shape=circle, inner sep=1.8pt] (A) at (0,0){$*$};
\draw[shift={(0,0.4)}] (0,0) node{\tiny $r$};
\node[draw, shape=circle, inner sep=1.8pt] (B) at (2,0){$*$};
\draw[shift={(0,0.4)}] (2,0) node{\tiny $r$};
\draw (A) --node[auto=left]{\tiny $\langle 2r \rangle$} (B);
\end{tikzpicture}
\caption{Cusp with parameter $r \ge 1$ and length $=2$}
\label{Cusp2}
\end{subfigure} 

\vspace{0.5cm}
\begin{subfigure}[b]{1\textwidth}
\centering
\begin{tikzpicture}
\node[draw, shape=circle, inner sep=1.8pt] (X) at (-1.5,0){$*$};
\draw[shift={(0,0.4)}] (-1.5,0) node{\tiny $r$};
\node[draw, shape=circle, inner sep=1.8pt] (A) at (0,0){$*$};
\draw[shift={(0,0.4)}] (0,0) node{\tiny $2r$};
\node[draw, shape=circle, inner sep=1.8pt] (B) at (1.5,0){$*$};
\draw[shift={(0,0.4)}] (1.5,0) node{\tiny $2r$};
\draw node (C) at (3,0){$\cdots$};
\node[draw, shape=circle, inner sep=1.8pt] (D) at (4.5,0){$*$};
\draw[shift={(0,0.4)}] (4.5,0) node{\tiny $2r$};
\node[draw, shape=circle, inner sep=1.8pt] (E) at (6,0){$*$};
\draw[shift={(0,0.4)}] (6,0) node{\tiny $r$};

\draw (X) --node[auto=left]{\tiny $\langle 2r \rangle$} (A);
\draw (A) --node[auto=left]{\tiny $\langle 2r \rangle$} (B);
\draw (B) --node[auto=left]{\tiny $\langle 2r \rangle$} (C);
\draw (C) --node[auto=left]{\tiny $\langle 2r \rangle$} (D);
\draw (D) --node[auto=left]{\tiny $\langle 2r \rangle$} (E);
\end{tikzpicture},
\caption{Twisted cusp with parameter $r \ge 1$}
\label{twisted cusp}
\end{subfigure} 

\begin{subfigure}[b]{1\textwidth}
\centering
\begin{tikzpicture}
\node[draw, shape=circle, inner sep=1.8pt] (A) at (0,0){$*$};
\draw[shift={(0,0.4)}] (0,0) node{\tiny $r$};
\node[draw, shape=circle, inner sep=1.8pt] (B) at (1.2,1.2){$*$};
\draw[shift={(0,0.4)}] (1.2,1.2) node{\tiny $r$};
\draw node (C) at (3,1.2){$\cdots$};
\node[draw, shape=circle, inner sep=1.8pt] (D) at (4.8 ,1.2){$*$};
\draw[shift={(0,0.4)}] (4.8,1.2) node{\tiny $r$};
\node[draw, shape=circle, inner sep=1.8pt] (E) at (6,0){$*$};
\draw[shift={(0,0.4)}] (6,0) node{\tiny $r$};
\node[draw, shape=circle, inner sep=1.8pt] (F) at (4.8,-1.2){$*$};
\draw[shift={(0,0.4)}] (4.8,-1.2) node{\tiny $r$};
\draw node (G) at (3,-1.2){$\cdots$};
\node[draw, shape=circle, inner sep=1.8pt] (H) at (1.2,-1.2){$*$};
\draw[shift={(0,0.4)}] (1.2,-1.2) node{\tiny $r$};

\draw (A) --node[auto=left]{\tiny $\langle r \rangle$} (B);
\draw (B) --node[auto=left]{\tiny $\langle r \rangle$} (C);
\draw (C) --node[auto=left]{\tiny $\langle r \rangle$} (D);
\draw (D) --node[auto=left]{\tiny $\langle r \rangle$} (E);
\draw (A) --node[auto=left]{\tiny $\langle r \rangle$} (H);
\draw (H) --node[auto=left]{\tiny $\langle r \rangle$} (G);
\draw (G) --node[auto=left]{\tiny $\langle r \rangle$} (F);
\draw (F) --node[auto=left]{\tiny $\langle r \rangle$} (E);
\end{tikzpicture}
\caption{Cusp with parameter $r \ge 1$ and length $ \ge 3$}
\label{cusp}
\end{subfigure}
\caption{List of dual graphs of non-rational numerical log canonical singularities}
\label{classification lc non rational}
\end{figure}
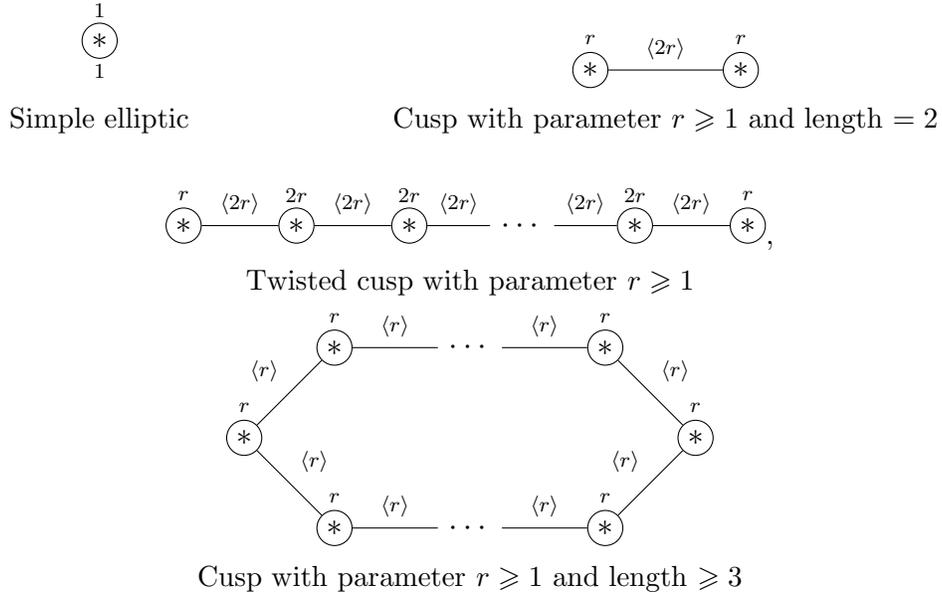


\end{document}